\documentclass[a4paper]{article}

\usepackage[english]{babel}

\usepackage[utf8]{inputenc}
\usepackage[utf8]{inputenc}
\usepackage{amsmath}
\usepackage{graphicx}
\usepackage{amssymb}
\usepackage{amsthm}
\usepackage{tikz-cd}
\usepackage{mathrsfs}
\usepackage[colorinlistoftodos]{todonotes}
\usepackage{enumitem}
\usepackage{yfonts}
\usepackage{ dsfont }
\usepackage{MnSymbol}
\usepackage{slashed}

\title{A gluing construction of collapsing Calabi-Yau metrics on K3 fibred 3-folds}

\author{Yang Li}

\date{\today}
\newtheorem{thm}{Theorem}[section]
\newtheorem{lem}[thm]{Lemma}

\theoremstyle{definition}

\newtheorem{cor}[thm]{Corollary}

\newtheorem{rmk}{Remark}
\newtheorem{prop}[thm]{Proposition}

\newtheorem*{Acknowledgement}{Acknowledgement}

\newcommand{\ie}{\emph{i.e.} }
\newcommand{\cf}{\emph{cf.} }

\newcommand{\R}{\mathbb{R}}
\newcommand{\C}{\mathbb{C}}
\newcommand{\Z}{\mathbb{Z}}

\newcommand{\norm}[1]{\left\lVert#1\right\rVert}
\newcommand{\Lap}{\Delta}

\DeclareMathOperator{\Tr}{Tr}

\begin{document}
	\maketitle

\begin{abstract}
We use the gluing method to give a refined description of the collapsing Calabi-Yau metrics on Calabi-Yau 3-folds admitting a Lefschetz K3 fibration.
\end{abstract}

\section{Introduction and background}

Let $X$ be a compact Calabi-Yau 3-fold with a Lefschetz K3 fibration $\pi: X\to Y=\mathbb{P}^1$. Given a reference K\"ahler metric $\omega_X$ on $X$ and $\omega_Y$ on $Y$, we aim to describe the collapsing family of Calabi-Yau metrics $\tilde{\omega}_{t}$ representing the K\"ahler class $[\omega_X+ \frac{1}{t}\pi^* \omega_Y]$ where $0<t\ll1$. Without loss of generality, we impose the volume normalisation $\int_{X_y} \omega_X^2=1$ where $X_y$ is any fibre of $\pi$, and $\int_Y \omega_Y=1$. Denote $S$ as the finite set of critical values of $\pi$. For simplicity, we assume each singular fibre contains only one nodal point.

Collapsing Calabi-Yau metrics for general fibrations have been studied from the viewpoint of a priori estimates, focusing mostly on the behaviour of $\tilde{\omega}_t$ away from the singular fibres $\pi^{-1}(S)$ (\cf e.g. \cite{Tosatti}). The basic picture (\ie the `semi-Ricci-flat' description) is that the collapsing metric involves two scales. If we scale down the family of metrics to $t\tilde{\omega}_t$ whose diameter scale $\sim 1$, then away from $S$ as $t\to 0$ these CY metrics collapse down to a limiting metric $\tilde{\omega}_Y$ on the base $Y$  called the generalised K\"ahler-Einstein metric, satisfying
\begin{equation}
\text{Ric}(\tilde{\omega}_Y)= \text{Weil Petersson metric}.
\end{equation}
This $\tilde{\omega}_Y$ for a general fibration has some singularity along $S$. If we keep the fibres normalised to volume 1, then away from the singular fibres, the fibrewise restrictions $\tilde{\omega}_t|_{X_y}$ will converge smoothly to the Calabi-Yau metric on $X_y$ in the class $[\omega_X|_{X_y}]$, and a tubular neighbourhood of $X_y$ will look like the metric product of $X_y$ with a flat Euclidean space.

More formally, one can introduce the semi-Ricci-flat metric $\omega_{SRF}$. We solve the Monge-Amp\`ere equation on the fibres $X_y$ to find a function $\psi_y$, such that $\omega_X|_{X_y}+\sqrt{-1}\partial \bar{\partial} \psi_y$ is the Calabi-Yau metric on $X_y$. Set $\psi=\psi_y$ as a global function on $X$, then we can write
\begin{equation}\label{semiflatmetric0}
\omega_{SRF}= \omega_X+\sqrt{-1} \partial \bar{\partial} \psi+ \frac{1}{t} \pi^* \tilde{\omega}_Y.
\end{equation}
Its defining feature is that its restriction to fibres are Ricci-flat, and as such captures key features of the collapsing metrics $\tilde{\omega}_t$. However, its definition involves the ambiguity of a form pulled back from the base, is not necessarily positive definite, and can be quite singular near $\pi^{-1}(S)$.

As we approach the nodal points of the fibration, then at a third length scale of order $\sim t^{1/6}$ (the `quantisation scale'), much smaller than the diameter scale of the fibres $\sim 1$, one observes that the semi-Ricci-flat description must break down \cite{Li}. This motivates the  construction of a model CY metric $\omega_{\C^3}$ on $\C^3$, whose asymptotic behaviour at infinity is designed to match up approximately with the semi-Ricci-flat metric \cite{Li}\cite{Ronan}\cite{Gabor}. It was further predicted that this model metric should arise as a scaling limit of $\tilde{\omega}_t$ near the nodal points, describing the geometry at the quantisation scale \cite{Li}.

The a priori estimate method is difficult to detect the geometry at extremely small length scales. On the other hand, the gluing method has been used to some effect in collapsing problems, such as Joel Fine's construction of cscK metrics on fibred complex surfaces \cite{Fine}, and Gross and Wilson's construction of CY metrics on elliptic K3 surfaces \cite{GrossWilson}. These works tend to rely on very favourable gluing models, such that the gluing error is already extremely small before the perturbation step.

The main result of this paper is to carry out the gluing construction for $\tilde{\omega}_t$ (\cf Theorem \ref{CYmetric}). As an immediate consequence,

\begin{thm}\label{GHlimitfibrescale}
As $t\to 0$, the family of CY metrics $\tilde{\omega}_t$ based at the nodal point converges in the Gromov-Hausdorff sense to the metric product $X_0\times \C$, where the nodal K3 fibre $X_0$ is equipped with the orbifold CY metric $\omega_{SRF}|_{X_0}$, and the $\C$ factor has the Euclidean metric.
\end{thm}

\begin{rmk}
This result has been obtained in \cite{Li2} by means of nonlinear estimates assuming a conjecture in pluripotential theory.
\end{rmk}

Morever, we verify that $\omega_{\C^3}$ arises as a blow up limit of the collapsing metrics $\tilde{\omega}_t$ near the node.

\begin{thm}\label{omegaCblowuplimit}
There exists a 1-parameter family of holomorphic embedding maps $F_t$ from large Euclidean balls in $\C^3$ to a neighbourhood of the node inside $X$, and a fixed number $A_0$ depending on the geometry of $\pi: X\to Y$, such that as $t\to 0$, the scaled CY metrics $( \frac{2A_0}{t})^{1/3} F_t^*\tilde{\omega}_t$ converge in $C^0_{loc}(\C^3)$ to the model CY metric $\omega_{\C^3}$.
\end{thm}

Roughly speaking, the metric ansatz is constructed by gluing the model metric $\omega_{\C^3}$ to the semi-Ricci-flat metric. One issue is that the semi-Ricci-flat metric istelf is expected to be singular on the singular fibre, and thus needs to be regularised first. The resulting metric ansatz suffers from rather large gluing errors, and one needs to work with rather coarse function spaces to perturb this into the actual CY metric $\tilde{\omega}_t$.

The key issue is to understand the harmonic analysis of the Laplace operator for the metric ansatz. The difficulty is the simultaneous presence of several scales with very different characteristic behaviours, an issue inherent in any collapsing problem and made more acute by the presence of singular fibres. The 
 technique is largely drawn from the work of G. Sz\'ekelyhidi \cite{Gabor}.  It involves analysing mapping properties of weighted H\"older spaces for every model geometry at each scale, decomposing the functions into pieces each sensitive only to one particular scale, inverting the Laplacian approximately on individual pieces using the various model Green operators, and patching the pieces to an approximate global solution. The main advantage of this method, aside from giving a fairly explict description of the Green operator, is that it allows us to derive a $t$-independent bound on a suitable operator norm, and in this sense this linear theory is optimal.

It is worth pointing out that following the recent works \cite{Ronan}\cite{Gabor}, many other new examples of complete CY metrics on $\C^n$ are now known, which are strong candidates for modelling collapsing fibrations with higher dimensional fibres. Such examples are likely to provide a vast generalisation of the main result of the present paper.

\begin{rmk}
All constants are uniform for sufficiently small $t$ unless stated otherwise.
\end{rmk}

\begin{Acknowledgement}
	The author is grateful to his PhD supervisor Simon Donaldson and co-supervisor Mark Haskins for their inspirations, Gabor Sz\'eklyhidi for discussions, and the Simons Center for hospitality. 
	
	This work was supported by the Engineering and Physical Sciences Research Council [EP/L015234/1], the EPSRC Centre for Doctoral Training in Geometry and Number Theory (The London School of Geometry and Number Theory), University College London. The author is also funded by Imperial College London for his PhD studies.
\end{Acknowledgement}

\section{Construction of metric ansatz}

\subsection{The generalised K\"ahler Einstein metric $\tilde{\omega}_Y$}\label{ThegeneralisedKahlerEinsteinmetric}

As mentioned in the introduction, the generalised K\"ahler Einstein metric $\tilde{\omega}_Y$ on the base $Y$ models the collapsing limit of the scaled family of CY metrics $t\tilde{\omega}_t$. Since the base $Y$ is complex one-dimensional, we can write down $\tilde{\omega}_Y$ rather more explicitly (\cf \cite{Tosatti}\cite{Li2}). Let $\Omega$ be the holomorphic volume form on $X$, normalised to $\int_X \sqrt{-1}\Omega\wedge \overline{\Omega}= 1$. Under our normalisation convention $\int_Y \omega_Y=\int_Y \tilde{\omega}_Y= 1$, this $\tilde{\omega}_Y$ is just the pushforward of the volume form
\begin{equation}\label{generalisedKEmetric}
\tilde{\omega}_Y= \pi_*( \sqrt{-1}\Omega\wedge \overline{\Omega} ).
\end{equation}
If we pick holomorphic local coordinate $y$ on $Y$, then we can write $\Omega=dy\wedge \Omega_y$, where by adjunction $\Omega_y$ is the holomorphic volume form on the fibre $X_y$, and when $y$ varies it gives a holomorphic section of the relative canonical bundle. The formula (\ref{generalisedKEmetric}) boils down to
\begin{equation}
\tilde{\omega}_Y= \sqrt{-1} dy\wedge d\bar{y} \int_{X_y} \Omega_y \wedge \overline{\Omega}_y=A_y \sqrt{-1} dy\wedge d\bar{y}.
\end{equation}

In our situation the only singularity in the fibration $\pi$ are assumed to be nodal. Then 
\begin{lem}
The function $A_y=\int_{X_y} \Omega_y \wedge \overline{\Omega}_y$ is Lipschitz in $y$.	
\end{lem}

\begin{proof}
We focus on the fibration $\pi:\mathcal{X}\to D_y$ over a small disc around the nodal fibre $X_0$, and take the square root fibration $\mathcal{X}'\to D_{\sqrt{y}}$. After taking a small resolution $\tilde{\mathcal{X}}'\to \mathcal{X}'$, the fibration $\tilde{\mathcal{X}}'\to D_{\sqrt{y}}$ 
becomes a submersion. Since $\Omega_y$ still defines a holomorphic section of the relative canonical bundle for the new fibration, the submersion property shows that $A_y$ is a smooth function over the square root base, namely $A_y$ is a smooth function of $\sqrt{y}$. But $A_y$ is an even function of $\sqrt{y}$, hence Lipschitz in $y$.

We take a closer examination at the singularity of $A_y$, which is not  needed for the proof. Notice that $\Omega_y$ is a closed 2-form on $X_y$, and so is $\overline{\Omega}_y$. This means the fibrewise integral $A_y$ only depends on the cohomology class of $\Omega_y$, which is the same data as the period integrals. Let $\Sigma\in H_2(X_y)$ be the class of a vanishing cycle. If a 2-cycle $\alpha\in H_2(X_y)$ is monodromy invariant, or equivalently it is orthogonal to $\Sigma$ under the intersection product, then the period integral $\int_{\alpha}\Omega_y$ is smooth in $y$, because we can make the representing cycles avoid the nodal point. To understand $\int_{\Sigma} \Omega_y$, we again pass to the family $\tilde{\mathcal{X}}'\to D_{\sqrt{y}}$. This vanishing cycle class $\Sigma$ becomes the class of the exceptional $\mathbb{P}^1$ when we take the small resolution. Again $\int_{\Sigma} \Omega_y$ is smooth in $\sqrt{y}$. Furthermore, Picard-Lefschetz formula implies that $\int_{\Sigma} \Omega_y$ is an odd function of $\sqrt{y}$, and the nature of period integrals implies this function is holomorphic in $\sqrt{y}$, so 
\[
\int_{\Sigma} \Omega_y=  g(y)\sqrt{y},
\]
where $g$ is a holomorphic function in $y$. The Lefschetz fibration imposes a further nondegeneracy condition on the deformation of the nodal fibre, which being translated into period integrals means $g(0)\neq 0$. Combining these discussions, the class $[\Omega_y]\in H^2(X_y)$ is the sum of a smooth monodromy invariant part and an orthogonal part $(g(y)\sqrt{y}) \Sigma$. This implies
\[
A_y = \int_{X_y} [\Omega_y]\wedge \overline{[\Omega_y]} =\text{smooth term}-2 |g(y)|^2 |y|.
\]
The factor $-2$ comes from $\Sigma\cdot \Sigma=-2$.
Notice $|g(y)|^2$ is smooth in $y$, but the modulus function $|y|$ is not $C^1$ in $y$, despite being smooth in $\sqrt{y}$. From the nondegeneracy condition $g(0)\neq 0$, we see the Lipschitz regularity is the sharp statement.
\end{proof}

\begin{rmk}
This failure of smoothness constrains the regularity of the metric ansatz we can produce.
\end{rmk}

Clearly $A_y\geq C> 0$, so $\tilde{\omega}_Y$ is uniformly equivalent to ${\omega_Y}$ on $Y$. We will later abuse notation to regard $\omega_Y$ and $\tilde{\omega}_Y$ also as forms on $X$.

\subsection{CY metrics on smoothings of the nodal K3 fibre}\label{CYmetricsonsmoothingsofthenodalK3fibre}

We now describe the CY metrics on the K3 fibres which are small deformations of any chosen nodal fibre. The basic picture is that these are obtained by gluing scaled versions of the Eguchi-Hanson metric to the CY metric on the nodal K3 fibre.
This section will be brief since there 
 are many gluing constructions of very similar nature in the literature, e.g. \cite{Donaldson}\cite{Spotti}, but we want to give enough details to keep track of the key estimates for later use.

On the nodal central fibre $X_0$, we write the orbifold CY metric as $\omega_{SRF}|_{X_0}=\omega_X|_{X_0} + \sqrt{-1} \partial \bar{\partial} \psi_0$, solving the Monge-Amp\`ere equation
\[
(\omega_X|_{X_0}+ \sqrt{-1} \partial \bar{\partial} \psi_0)= \frac{1}{A_{0}}   \Omega_{0}\wedge \overline{\Omega}_{0}, \quad \int_{X_0} \psi_0 \omega_X^2=0, \quad A_0=A_{y}|_{y=0}, \Omega_0=\Omega_y|_{y=0}.
\]
At the nodal point $P$, 
the fibration $\pi$ induces (up to scale) a complex symmetric bilinear form on the tangent space $T_P X$, or equivalently an $SO(3, \C)\simeq SL(2, \C)/\Z_2$ structure. The orbifold CY metric singles out a Hermitian metric on $\C^2/\Z_2$, or equivalently an $SU(2)/\Z_2=SO(3, \R)$ structure, so we have a preferred Hermitian structure $|\cdot |$ on $T_P X$. We can then choose local coordinates $\mathfrak{z}_1, \mathfrak{z}_2, \mathfrak{z}_3$ on an open neighbourhood $U_1\subset X$ and a local coordinate $y$ on $Y$, where the fibration $\pi$ is represented by $y=\mathfrak{z}_1^2+ \mathfrak{z}_2^2+ \mathfrak{z}_3^2$, and $\psi_0=c_0+ r^2+ O(r^4)$ with $r= ( |\mathfrak{z}_1|^2+ |\mathfrak{z}_2|^2+ |\mathfrak{z}_3|^2 )^{1/4}$.
Here the innocuous constant $c_0$ is a matter of normalisation, which appears because we impose $\int_{X_0} \psi_0 \omega_X^2=0$. Using the compatibility condition $dy\wedge \Omega_y=\Omega $ at $y=0$, one calculates the normalisation on the holomorphic volume form at the nodal point to be
\[
\begin{split}
\sqrt{-1}\Omega\wedge \overline{\Omega}|_{\mathfrak{z}=0}&=
\Omega_0\wedge \overline{\Omega}_0\wedge \sqrt{-1}dy\wedge d\bar{y}|_{\mathfrak{z}=0}
=
\sqrt{-1}A_{0}(\sqrt{-1}\partial \bar{\partial} r^2)^2\wedge dy\wedge d\bar{y}|_{\mathfrak{z}=0} \\
&=A_{0}\prod_i \sqrt{-1} d\mathfrak{z}_i d\bar{\mathfrak{z}}_i,
\end{split}
\]
so up to a constant in $U(1)$, the holomorphic volume form is locally given in $U_1$ by $\Omega=\sqrt{A_{0}}d\mathfrak{z}_1 d\mathfrak{z}_2 d\mathfrak{z}_3(1+ O(\mathfrak{z}))$.

We focus on the fibres over the small local base $\{ |y|<\epsilon_1  \}$ with $\epsilon_1\ll1$. For convenience, we extend the function $r$ on $X_y\cap U_1$ smoothly to the whole $X_y$, such that outside the coordinate neighbourhood $r$ is of order 1. In $X_y\cap U_1$, one has the scaled Eguchi-Hanson metric $EH_y$ given by $\sqrt{-1}\partial \bar{\partial} \sqrt{ r^4+ |y|}$. Then the function $r$ can be thought as a smoothed out version of the distance to the vanishing cycle. This allows us to define the weighted H\"older spaces $C^{k,\alpha}_\beta(X_y)$ on $X_y$.  The weighted H\"older norm of a function $f$ on $X_y$ can be defined by 
\[
\begin{split}
\norm{f}_{C^{k,\alpha}_\beta}=&\norm{f}_{C^{k,\alpha} (X_y\setminus \{r>c\}), \omega_X) } + 
\sum_{j\leq k} \sup_{X_y\cap U_1} r^{ -\beta+j} |\nabla^j_{EH_y} f| \\ +& \sup_{d_{EH_y}(x,x')\ll r(x), x, x'\in X_y\cap U_1} r(x)^{-\beta+\alpha+k} \frac{ |\nabla^k_{EH_y} f (x)- \nabla^k_{EH_y} f (x')|}{d_{EH_y}(x,x')^{\alpha}  },
\end{split}
\]
where the difference of two tensors at nearby points are compared by parallel transport along the unique minimal geodesic joining them. The constant $c$ is meant to be small enough to make $\{r<c\}$ contained in the coordinate neighbourhood $U_1$. Similarly, one can define the $C^{k}_{\beta}(X_y)$ norm by setting $\alpha$ to zero, and it is easy to extend these definition to tensors. A useful feature of H\"older norms, which will be used repeatedly later, is that they are local, namely if the manifold is covered by several regions with some overlap, then it suffices to estimate the H\"older norms on each individual region.

Let $1\ll\Lambda_1\ll \epsilon_1^{-1/4}$ be a large number. Then for any $y$ with $|y|<\epsilon_1$, we can find a diffeomorphism $G_{0,y}$ between $X_y \setminus \{ r< \Lambda_1 |y|^{1/4}   \}$ (namely the complement of a neighbourhood of the vanishing cycle) and an open subset of $X_0\setminus \{ r< \frac{1}{2} \Lambda_1 |y|^{1/4}   \}$. We can demand these diffeomorphisms to depend smoothly on $y$, namely they fit into a fibration preserving diffeomorphism
\begin{equation}\label{diffeomorphismG0}
G_0:   \{ x\in X:  |y|< \epsilon \} \setminus \{  r< \Lambda_1 |y|^{1/4}   \}  \to
U'\subset X_0 \times \{|y|<\epsilon_1  \} .
\end{equation}
(This can be defined, for instance, by flowing along the  vector fields orthogonal to the fibres under the $\omega_X$ metric. Or one can prescribe the diffeomorphism explicitly in the coordinate neighbourhood $U_1$ and try to extend it outside $U_1$, similar to \cite{Spotti}. Many reasonable constructions will satisfy the desired estimates.) The diffeomorphism $G_{0,y}$ is approximately holomorphic: we can arrange so that on $X_y\setminus \{r<\Lambda_1 |y|^{1/4} \}$ where $G_{0,y}$ is defined, 
\[
|\Omega_y- G_{0,y}^* \Omega_0|_{\omega_X} \leq \frac{C|y|}{|\mathfrak{z}|^2} |\Omega_y|_{\omega_X}= \frac{C|y|}{r^4} |\Omega_y|_{\omega_X}.
\]
That is, the variation of complex structure causes an error of order $O(\frac{|y|}{r^4})$.
Furthermore, 
we can compare potentials, holomorphic volume forms and the background metric $\omega_X$ on $X_0$ and $X_y$ to higher order. For instance,
\begin{equation}\label{CYmetriconXyintermediate1}
\begin{split}
& \norm{G_{0,y}^* r^2- r^2}_{  C^k_{-2}( (U_1\cap X_y)\setminus \{r<\Lambda_1 |y|^{1/4} \} )  } \leq C|y|,
\\
& \norm{ G_{0,y}^* \Omega_0 - \Omega_y }_{C^k_{-4}(  X_y\setminus \{r<\Lambda_1 |y|^{1/4} \} )  } \leq C|y|. \\
&
\norm{G_{0,y}^* (\omega_X|_{X_0})- \omega_X|_{X_y} }_{ C^k_{-2}( X_y\setminus \{r<\Lambda_1 |y|^{1/4} \} )   } \leq  C|y|.
\end{split}
\end{equation}
The various power law behaviours can be seen quite easily from dimensional analysis. It's enough to examine what happens inside the coordinate neighbourhood $U_1\subset \C^3$. The point is that to the leading order, expressions like $r^2$, $\Omega_y$ and $\omega_X$ have some homogeneity behaviour under the scaling $\mathfrak{z}\to \lambda \mathfrak{z}$, and the diffeomorphism $G_{0,y}$ would approximately respect this homogeneity, so the problem reduces by scaling to the case with
 $|\mathfrak{z}|\sim 1, |y|\ll1$, where estimates of the above type are clear. This type of arguments will be tacitly used many times later when we assert good properties about diffeomorphisms.

One can now construct an approximate CY metric $\omega'_y$, essentially by gluing the Eguchi-Hanson metric $EH_y$ to $G_{0,y}^* (\omega_{SRF}|_{X_0})$ at scale $r\sim |y|^{1/6}$. The gluing region is then contained in the coordinate neighbourhood $U_1$ because $\epsilon_1\ll1$, and avoids the vicinity of the vanishing sphere $\{r< \Lambda_1|y|^{1/4}\}$.

Let $\gamma_1(s)$ be a cutoff function,
\[
\gamma_1(s)= \begin{cases}
1 \quad \text{if } s>2, \\
0 \quad \text{if } s<1.
\end{cases}
\]
and let $\gamma_2=1-\gamma_1$.
We define 
\begin{equation}\label{CYmetriconXyapproximate}
\omega'_y= \omega_X|_{X_y} + \sqrt{-1} \partial \bar{\partial} \{\gamma_1( \frac{r}{|y|^{1/6}}  ) G_{0,y}^* (\psi_0 -c_0) + \gamma_2( \frac{r}{|y|^{1/6}} )\sqrt{ r^4+|y|}  \}
\end{equation}
The diffeomorphism is well defined on the support of $\gamma_1$, so this expression makes sense. We remark that there can be many minor variants to the gluing ansatz.

\begin{lem}\label{CYmetriconXyintermediatelemma}
If $\epsilon_1\ll1$, then $\omega_y'$ is positive definite, namely a K\"ahler metric, and for $-2\leq\beta<0$ satisfies
\begin{equation*}
\omega_y'^2= \frac{1}{A_y} (1+f_y') \Omega_y\wedge \overline{\Omega}_y , \quad \norm{f_y'}_{  C^{k,\alpha}_{\beta-2}(X_y)        }\leq C(k, \alpha) |y|^{-\frac{1}{6} \beta+ \frac{2}{3}  }.
\end{equation*}
Here the constants are independent of $y$ as long as $|y|<\epsilon_1$. 
\end{lem}

\begin{rmk}
In particular $|f_y'|_{L^\infty}=O(|y|^{\frac{1}{12}(\beta+2)} )\ll1$, meaning that the nonlinear effect is weak.	
\end{rmk}

\begin{proof}
When $r> 2|y|^{1/6}$, this $\omega_y'$ is just 
$\omega_X|_{X_y} + \sqrt{-1} \partial \bar{\partial}  G_{0,y}^* \psi_0$, which we would like to compare to $G_{0,y}^* ( \omega_X|_{X_0}+ \sqrt{-1} \partial \bar{\partial} \psi_0    )$. To control their difference, we examine 
\[
r^2| \partial \bar{\partial}G_{0,y}^* (\psi_0-r^2-c_0) - G_{0,y}^* \partial \bar{\partial}(\psi_0-r^2-c_0)|_{EH_y}
\leq  \frac{C|y|}{r^4}
 \sum_{j=1}^2 |r^j\nabla^{j}  (\psi_0-r^2-c_0)|            
 \leq C|y|.
\]
The first inequality uses the general observation that the relative error caused by variation of complex structure is of order $O(\frac{|y|}{r^4})$, and the second uses that $|\nabla^k (\psi-c_0-r^2)|=O(r^{4-k})$ on the orbifold $X_0$. This can be contrasted with
\[
r^2| \partial \bar{\partial}G_{0,y}^* r^2 - G_{0,y}^* \partial \bar{\partial}r^2|_{EH_y}
\leq  \frac{C|y|}{r^2} ,
\]
which is the dominant error term for small $r$.
The higher order estimates proceed in the same fashion, and one evantually gets
\[
\norm{ \omega_X|_{X_y} + \sqrt{-1} \partial \bar{\partial}  G_{0,y}^* \psi_0 -   G_{0,y}^* ( \omega_X|_{X_0}+ \sqrt{-1} \partial \bar{\partial} \psi_0    )  }_{ C^k_{-4} (X_y \cap \{  r> 2|y|^{1/6}  \}   ) }  \leq C|y|.
\]
This easily implies the positive definiteness of $\omega_y'$ in this region.
Morever, 
\[
\norm{ \omega_y'^2 - G_{0,y}^* \omega_{SRF}|_{X_0}^2 }_{ C^k_{-4} (X_y \cap \{  r> 2|y|^{1/6}  \}   ) }   \leq C|y|. 
\]
But we know
\[
\omega_{SRF}|_{X_0}^2=A_0^{-1} \Omega_0 \wedge \overline{\Omega}_0, \quad |A_0-A_y|\leq C|y|, \quad \norm{\Omega_y- G_{0,y}^*\Omega_0}_{ C^k_{-4} (X_y \cap \{  r> 2|y|^{1/6}  \}) } \leq C|y|,
\]
so we can assemble the facts to see
\begin{equation}\label{CYmetriconXyintermediate2}
\norm{\omega_y'^2 - A_y^{-1} \Omega_y\wedge \overline{\Omega}_y }_{ C^k_{-4} (X_y \cap \{  r> 2|y|^{1/6}  \})  } \leq C|y| .
\end{equation}

Now we analyse the region $\{|y|^{1/6}< r< 2|y|^{1/6}\}$, where the cutoff error is supported. The term $\omega_{X}|_{X_y}$ is of order $O(r^2)$ small compared to the Eguchi-Hanson metric $EH_y$. To understand the deviation of $\omega_y'$ from $EH_y$, it suffices to examine
 $\sqrt{-1}\partial \bar{\partial} \{ \gamma_1( \frac{r}{|y|^{1/6}}) ( G^*_{0,y}(\psi_0- c_0) - \sqrt{r^4+|y|}         )  \}$. 
We have 
\[
\begin{split}
& \norm{r^2- \sqrt{r^4+|y|} }_{C^k_{-2}(X_y\cap U_1)} \leq C|y|, \\
& \norm{ G_{0,y}^* r^2- r^2}_{C^k_{-2}((U_1\cap X_y) \setminus \{ r< \Lambda_1 |y|^{1/4}  \}  )   } \leq C|y|, \\
& \norm{ G_{0,y}^* (\psi_0-c_0- r^2)}_{C^k_{-2}(U_1\cap X_y) \setminus \{ r< \Lambda_1 |y|^{1/4}  \}  )   } \leq Cr^2 r^4\leq C|y|,
\end{split}
\]
so $\norm{ G^*_{0,y}(\psi_0- c_0) - \sqrt{r^4+|y|}     }_{ C^k_{-2}(\{ |y|^{1/6} < r< 2|y|^{1/6}\}    )  }\leq C|y|$. Using also $\norm{\gamma_1(\frac{r}{|y|^{1/6}}  )}_{C^k_0(X_y) } \leq C $,
we see
\[
\norm{      \partial \bar{\partial} \{ \gamma_1( \frac{r}{|y|^{1/6}}) ( G^*_{0,y}(\psi_0- c_0) - \sqrt{r^4+|y|}         )  \} }_{C^k_{-4}(\{ |y|^{1/6} < r< 2|y|^{1/6}\} )   } \leq C|y|.
\]
In particular, there is a pointwise estimate
\[
|\partial \bar{\partial} \{ \gamma_1( \frac{r}{|y|^{1/6}}) ( G^*_{0,y}(\psi_0- c_0) - \sqrt{r^4+|y|}         )  \} |_{EH_y} \leq  C|y| |y|^{-4/6} =O(|y|^{1/3})=O(r^2).
\]
We observe $r\sim |y|^{1/6}$ is precisely the scale where various error sources are of comparable strength. We can now easily see the positive definiteness of $\omega_y$ in this region.

There is yet another source of error coming from the holomorphic volume form. Since $\Omega= \sqrt{A_0}(1+O(\mathfrak{z})) d\mathfrak{z}_1 d\mathfrak{z}_2 d\mathfrak{z}_3$ where $O(\mathfrak{z})$ is a holomorphic function, and $\Omega=\Omega_y\wedge dy$, we can check from the explicit volume form of the Eguchi-Hanson metric, that 
\[
| \nabla^k_{EH_y} \{ (\sqrt{-1} \partial \bar{\partial} \sqrt{r^4+ |y|} )^2- A_y^{-1} \Omega_y \wedge \overline{ \Omega}_y \}  |_{EH_y} =O( r^{2-k})
\]
Combining these discussions,
\begin{equation}\label{CYmetriconXyintermediate3}
\norm{\omega_y'^2 - A_y^{-1} \Omega_y\wedge \overline{\Omega}_y }_{ C^k_{-4} (X_y \cap \{ |y|^{1/6}< r< 2|y|^{1/6}  \}  } \leq C|y| .
\end{equation}

Finally, when $r< |y|^{1/6}$, the error to the volume form is of order $O(r^2)$ with good higher order estimates. These combined with (\ref{CYmetriconXyintermediate2}), (\ref{CYmetriconXyintermediate3}) imply the claim.
\end{proof}

To perturb the approximate metric $\omega_y'$ into the actual CY metric $\omega_{SRF}|_{X_y}$ on $X_y$, we need the crucial mapping property of the Laplacian on the weighted H\"older spaces.

\begin{lem}(Compare \cite{Spotti} Proposition 3.2)\label{CYmetriconXyHoldermappingproperty}
If $-2<\beta<0$ and $|y|<\epsilon_1$, then the Laplacian $\Lap_{\omega_y'}: C^{k+2,\alpha}_\beta(X_y) \to C^{k,\alpha}_{\beta-2}(X_y)$ restricted to the subspaces of functions with $\int_{X_y} f\omega_y'^2=0$ is an isomorphism, and the inverse satisfies a uniform estimate in $f$ and $y$
\begin{equation}
\norm{\Lap_{ \omega_y'}^{-1} f}_{C^{k+2,\alpha}_\beta(X_y) } \leq
C(k, \alpha, \beta) \norm{ f}_{C^{k,\alpha}_{\beta-2}(X_y) }. 
\end{equation}
\end{lem}

\begin{rmk}
This can be proved using the weighted Schauder estimates
\[
\norm{u}_{  C^{k+2,\alpha}_\beta(X_y) }    \leq C \norm{\Lap u}_{ C^{k,\alpha}_{\beta-2}(X_y)   } + C\norm{r^{-\beta} u}_{ L^\infty  }
\]
and a standard blow up argument. 
\end{rmk}

The implicit function theorem then implies in a standard fashion that

\begin{prop}\label{CYmetriconXy}
(CY metrics on the smoothing of the nodal K3 fibre)
Let $-2<\beta<0$, and $|y|<\epsilon_1\ll1$.
There is a unique potential function $\psi_y'$ with $\int_{X_y} \psi_y' \omega_y'^2=0$, such that
\[
\omega_{SRF}|_{X_y}=\omega_y'+ \sqrt{-1} \partial \bar{\partial} \psi_y', \quad (\omega_{SRF}|_{X_y})^2= A_y^{-1} \Omega_y \wedge \overline{\Omega}_y,
\]
with the uniform estimate in $y$,
\begin{equation}
\norm{ \psi_y'}_{ C^{k+2, \alpha}_\beta   } \leq C(k, \alpha, \beta) |y|^{-\frac{1}{6}\beta+ \frac{2}{3} }.
\end{equation}
\end{prop}

\begin{rmk}
In particular
 $|\psi_y'|\leq C |y|^{\frac{2}{3} -\frac{1}{6}\beta }r^\beta $, so 
$|\int_{X_y} \psi_y' \omega_X^2| =O(|y|^{\frac{2}{3} -\frac{1}{6}\beta   })$.
There is a different normalisation convention for the potential,
\begin{equation}\label{semiflatmetric1}
\omega_{SRF}|_{X_y} = \omega_X + \sqrt{-1} \partial \bar{\partial} \psi_y, \quad \int_{X_y} \psi_y \omega_X^2=0.
\end{equation}
The advantage of (\ref{semiflatmetric1}) is that it makes sense also for fibres outside $\{|y|<\epsilon_1\}$, so is more useful for the global construction of the semi-Ricci-flat metric.
We have $\psi_y= \psi_y'+ \gamma_1( \frac{r}{|y|^{1/6}}  ) G_{0,y}^* (\psi_0 -c_0) + \gamma_2( \frac{r}{|y|^{1/6}} )\sqrt{ r^4+|y|} +c_0 + c_0'(y) $, where $c_0'(y)$ is a constant on $X_y$ with $|c_0'(y)|\leq C(\beta)|y|^{ \frac{2}{3}- \frac{1}{6}\beta  }$, for any $-2<\beta<0$ and $|y|<\epsilon_1$. 
\end{rmk}

We take the opportunity to consider deformation of the CY metrics $\omega_{SRF}|_{X_y}$ as the complex structure varies with $y$. When $|y|\geq \epsilon_1$ so $X_y$ is bounded away from the singular fibre, then it is a standard fact that the potential $\psi_y$ solving (\ref{semiflatmetric1}) deforms smoothly with $y$. In particular
we can take a trivialisation $G_{y'}$ for the fibration around a given fibre $X_{y'}$, which induce diffeomorphisms  $G_{y',y}$  identifying sufficiently nearby fibres $X_{y}$ with $X_{y'}$, and then the potentials are compared as $|\psi_y-\psi_{y'}|\leq C|y-y'|$. We would like to extend this kind of Lipschitz bound to fibres with $|y'|<\epsilon_1$.

Given a fibre $X_{y'}$ with $|y'|< \epsilon_1$, we can take a fibration preserving trivialisation $G_{y'}$ over the disc $\{ y: |y-y'|\leq \epsilon_2 |y'|  \}$,
\begin{equation}\label{diffeomorphismGy}
G_{y'}:  \{ x\in X: |y'-\pi(x)|\leq \epsilon_2 |y'|  \}  \to
X_{y'} \times \{|y-y'|\leq\epsilon_2 |y'|  \} \subset X_{y'}\times \C .
\end{equation}
This can be defined, for example, by flowing along the vector field
obtained by the orthogonal horizontal lift of tangent vector fields on $Y$, using an ambient metric $\omega_X$. For $|y-y'|\leq \epsilon_2 |y'|$,
this induces the diffeomorphisms $G_{y',y}$ from $X_y$ to $X_{y'}$, depending smoothly on $y$.
 We can demand
\[
|\Omega_{y}- G^*_{y',y}\Omega_{y'}|_{\omega_X} \leq \frac{ C|y-y'|} {  |\mathfrak{z}|^2   }|\Omega_{y}|_{\omega_X}= \frac{C|y-y'|}{ r^4  }|\Omega_{y}|_{\omega_X},
\]
namely the variation of complex structure causes errors of order $O(\frac{|y-y'|}{r^4})$. The analogue of (\ref{CYmetriconXyintermediate1}) is
\[
\begin{split}
& \norm{ G_{y',y}^* \Omega_{y'} - \Omega_{y} }_{C^k_{-4}(  X_y )  }  \leq C|y-y'|. \\
&
\norm{G_{y',y}^* (\omega_X|_{X_{y'}})- \omega_X|_{X_{y}} }_{ C^k_{-2}( X_y )   } \leq  C|y-y'|.
\end{split}
\]
We put an approximate CY metric on $X_{y}$ as
\[
\omega_{y}''= \omega_X|_{X_{y}}+ \sqrt{-1} \partial \bar{\partial} G_{y',y}^* \psi_{y'}.
\]
This can be compared to $G_{y',y}^*(\omega_{SRF}|_{X_{y'}})= G_{y',y}^*(\omega_X|_{X_{y'}} + \sqrt{-1}\partial \bar{\partial} \psi_{y'}  )$. To estimate their difference, the main issue is to control the norm of $G_{y',y}^*( \partial \bar{\partial} \psi_{y'}  )-\partial \bar{\partial} G_{y',y}^* \psi_{y'}$. We first examine the pointwise bound measured against $\omega_{SRF}|_{X_{y}}$.
\[
r^2| G_{y',y}^*( \partial \bar{\partial} \psi_{y'}  )-\partial \bar{\partial} G_{y',y}^* \psi_{y'}| \leq \frac{C|y-y'|}{r^4} 
\sum_{j=1}^2 |r^j\nabla^{j}  \psi_{y'}  |            
\leq \frac{C|y-y'|}{r^2}.
\]
The first inequality uses that the relative error caused by the variation of complex structure is of order $O(\frac{|y-y'|}{r^4} )$, and 
the second inequality makes use of Proposition \ref{CYmetriconXy} and its ensuing Remark to control $\psi_{y'}$.
This estimate can easily be improved to higher orders, to give
\[
\norm{
	\omega''_{y}- G_{y',y}^* (\omega_{SRF}|_{X_{y'}}) }_{C^k_{-4}(X_{y} )   } \leq C|y-y'|.
\]
One can assemble the facts to show for $-2\leq \beta<0$,
\[
\norm{  (\omega''_{y})^2- A_{y}^{-1} \Omega_{y} \wedge \overline{\Omega}_{y} }_{ C^{k, \alpha}_{\beta-2}(X_y)   } \leq C|y-y'||y'|^{-\frac{1}{4}(\beta+2)   }   .
\]
In particular the volume error is $O(\frac{|y-y'|}{r^4})$ small in $L^\infty$ norm, which for $|y-y'|\leq \epsilon_2 |y'| \leq C\epsilon_2 r^4 \ll r^4 $ is small in absolute norm. This signifies that nonlinear effect is weak. Then one can use Lemma \ref{CYmetriconXyHoldermappingproperty} and the implicit function to solve
\[
(\omega''_{y}+ \sqrt{-1} \partial \bar{\partial} \psi_{y}'')^2= A_{y}^{-1} \Omega_{y} \wedge \overline{\Omega}_{y},
\]
with estimate $\norm{\psi_{y}''}_{C^{k+2,\alpha}_\beta(X_{y})} \leq C(k,\alpha, \beta) |y-y'| |y'|^{ -\frac{1}{4}(\beta+2)  } $ for $-2<\beta<0$. Comparing this with
(\ref{semiflatmetric1}), and installing the suitable integral normalisation condition, we get

\begin{lem}\label{Lipschitztypeestimatepotentialfunction}
The function $\psi_y$ defined by (\ref{semiflatmetric1}) satisfies the Lipschitz type estimate: for $-2<\beta<0$,
\begin{equation*}
\norm{ G_{y',y}^*\psi_{y'} -\psi_{y} }_{ C^{k+2,\alpha}_\beta(X_{y})  } \leq C(k,\alpha, \beta) |y-y'| |y'|^{ -\frac{1}{4}(\beta+2)  } 
\end{equation*}
uniformly for $|y'|<\epsilon_1, |y-y'|\leq \epsilon_2 |y'|$.
\end{lem}

\begin{rmk}
As mentioned before, for $|y'|\geq \epsilon_1, |y-y'|\leq \epsilon_2|y'|$, we have the easier analogue: for $-2<\beta<0$,
\[
\norm{G^*_{y',y} \psi_{y'}- \psi_{y} }_{C^{k+2,\alpha}(X_{y})} \leq C(k,\alpha, \beta) |y-y'|.
\]
Here we can use the usual H\"older norm, and it is understood that $y,y'\in Y$ do not come close to other critical values in $S$. When $y, y'$ go beyond the coordinate neighbourhood, then $|y-y'|$ is replaced by the qualitatively similar expression $d_{\omega_Y}(y,y')$.
\end{rmk}

\begin{rmk}\label{Scaleofdeformation}
Comparing Proposition \ref{CYmetriconXy} and Lemma \ref{Lipschitztypeestimatepotentialfunction}, if we consider $|y-y'|\sim t^{1/2}$, then $\psi_y'$ and $G_{y',y}^*\psi_{y'}- \psi_{y}$ have comparable norm estimates when $|y|\sim t^{ \frac{6}{14+\beta}  }$.
\end{rmk}

\subsection{Geometry of the model metric $\omega_{\C^3}$ }\label{GeometryofthemodelmetricC3}

We give a quick review of the model CY metric $\omega_{\C^3}$ on $\C^3$, based on \cite{Li}\cite{Gabor}. Let $\C^3$ be equipped with the standard coordinates $z_1, z_2, z_3$ and a Hermitian structure $|\cdot|$. Define the functions 
\[
\begin{cases}
R=( |z_1|^2+ |z_2|^2+ |z_3|^2)^{1/4},	\\
\tilde{y}=z_1^2+z_2^2+z_3^2, \\
\rho=\sqrt{ |\tilde y|^2 + \sqrt{ R^4+1} }
\end{cases}
\]
Here $\tilde{y}$ gives the structure of the standard Lefschetz fibration on $\C^3$ over $\C_{\tilde{y}}$. Then there exists a CY metric $\omega_{\C^3}= \sqrt{-1} \partial \bar{\partial} \phi_{\C^3}$ on $\C^3$, with volume normalisation
\[
\omega_{\C^3}^3= \frac{3}{2} \prod_{i=1}^3 \sqrt{-1} dz_i \wedge d\bar{z}_i,
\]
and the leading order asymptote at infinity is given by
\begin{equation}
\phi_{\infty}= \frac{1}{2} |\tilde{y}|^2 + \sqrt{ R^4+ \rho      }, \quad \phi_{\C^3}=\phi_\infty+ \phi_{\C^3}'.
\end{equation}
Outside $\{|z|<1\}$ the function $\rho$ is uniformly equivalent to the $\omega_{\C^3}$-distance to the origin.  The distance to the vanishing cycles $\{ R^4=|\tilde{y}| \}$ is controlled by the function $R$ away from a large compact set, and the sizes of the vanishing cycles grow as $O(|\tilde{y}|^{1/4})$.

One can understand the asymptotic metric $\sqrt{-1}\partial \bar{\partial} \phi_\infty$ as follows. The term $\frac{1}{2}|\tilde{y}|^2$ pulls back the potential of the Euclidean metric on $\C_{\tilde{y}}$. This contribution is the dominant term for the horizontal component of the metric. When restricted to the fibres of the Lefschetz fibration,  there is the term  $\sqrt{ R^4+ \rho}$. This is an approximation to the potential of the Eguchi-Hanson metric on the fibre, which is $\sqrt{R^4+ |\tilde{y}|}$. Thus $\sqrt{-1}\partial \bar{\partial}\phi_\infty$ can be viewed as a regularised version of a semi-Ricci-flat metric.

The metric $\omega_{\C^3}$ exhibits 3 different characteristic behaviours.
It is a complete Ricci flat metric with singular tangent cone at infinity $\C^2/\Z_2\times \C$. The singular line $\{0\}\times \C$ of the tangent cone corresponds roughly to the vicinity of the vanishing cycles. However, if we place a sequence of points on the vanishing cycles, scale down $\omega_{\C^3}$ by a factor of  $|\tilde{y}|^{1/4}$, and let $\tilde{y}$ move to infinity, then the pointed Gromov-Hausdorff limit is $EH_1\times \C$, where $EH_1$ is the standard Eguchi-Hanson metric. On the other hand, inside the ball $\{|z|<1\}$ the metric $\omega_{\C^3}$ is uniformly equivalent to the Euclidean metric $\sqrt{-1}\sum dz_i\wedge d\bar{z}_i $.

We now follow \cite{Gabor} to define the double weighted H\"older space $C^{k, \alpha}_{\delta, \tau}(\C^3, \omega_{\C^3})$  taylored to this mixture of behaviours. Let $\kappa$ be a fixed small positive number, and $K$ be a fixed large number. We define a weight function $w$ by
\[
w=\begin{cases}
1 \quad\quad & \text{if } R\geq \kappa \rho, \\
\frac{ R}{ \kappa \rho} & \text{if } R \in ( \kappa^{-1} \rho^{1/4}, \kappa \rho       ),
\\
\kappa^{-2}\rho^{-3/4} & \text{if } R\leq  \kappa^{-1} \rho^{1/4}.
\end{cases}
\]
The H\"older seminorm of a tensor $T$ is given by
\[
[T]_{0,\alpha}=\sup _{\rho(z)>K  } \rho(z)^\alpha w(z)^\alpha \sup_{z\neq z', z'\in B(z,c R(z))} \frac{  |T(z)- T(z')|}{d(z,z')^\alpha }.
\]
Here $c>0$ is such that the metric balls $B(z, cR(z))$ have bounded geometry and are geodesically convex, so we can compare $T(z)$ with $T(z')$ using parallel transport along a geodesic.
The weighted norm of a function $f$ is then defined by
\begin{equation*}%\label{HoldernormonC3}
\norm{f}_{C^{k,\alpha}_{\delta, \tau } }= \norm{f}_{C^{k,\alpha}(\rho<2K) } +\sum_{j=0}^k \sup_{\rho(z)>K} \rho^{-\delta+j} w^{-\tau+j} |\nabla^j f| +[\rho^{-\delta+k} w^{-\tau+k} \nabla^k f ]_{0,\alpha}.
\end{equation*}
Then the deviation $\phi_{\C^3}'$ of $\phi_{\C^3}$ from its asymptotic expression $\phi_\infty$ is (\cf Proposition 6.9 \cite{Li} for a more refined version, which extracts the leading term in $\phi_{\C^3}'$)
\begin{equation}\label{phiC3error}
\norm{\phi_{\C^3}'}_{C^{k,\alpha}_{\delta, 0}  } \leq C(k, \alpha, \delta),   \quad  \forall \delta>-1.
\end{equation}

\begin{rmk}
It is an essential prerequisite for our main gluing construction that the model metric $\omega_{\C^3}$ is unique in its asymptotic class; more precisely, if $\phi_{\C^3, 1}'$ and $\phi_{\C^3,2}'$ both satisfy the bound (\ref{phiC3error}) for some $\delta<0$, and $\phi_{\C^3, i}= \phi_{\infty}+ \phi_{\C^3,i}'$ for $i=1,2$ both have the same Calabi-Yau volume form, then $\phi_{\C^3,1}=\phi_{\C^3,2}$. To see this, set $u= \phi_{\C^3,1}'-\phi_{\C^3,2}'$, then standard integration by part argument shows
\[
\int_{\C^3} |\nabla |u|^p |^2 \omega_{\C^3}^3=0, \quad p\gg1.
\]
Here we crucially need the decay property of $u$ at infinity to drop boundary terms. It remains an interesting question what is the most general class of potentials for which one can prove uniqueness.
\end{rmk}

We next describe (heuristically) how this model metric on $\C^3$ fits into $X$. From Section \ref{CYmetriconXy}, we see that in $U_1\cap \{|y|< \epsilon_1 \}$, the Calabi-Yau metrics on fibres $\omega_{SRF}|_{X_y}$ are approximately the Eguchi-Hanson metrics. From Section \ref{ThegeneralisedKahlerEinsteinmetric} the generalised KE metric is $\tilde{\omega}_Y=A_y \sqrt{-1}dy\wedge d\bar{y}$. Thus the semi-Ricci-flat metric is approximately (\cf (\ref{semiflatmetric0}))
\[
\omega_{SRF}\sim \sqrt{-1}\partial \bar{\partial} \sqrt{ r^4+ |y|}+ \frac{1}{t} \tilde{\omega}_Y \sim \sqrt{-1}\partial \bar{\partial} (\sqrt{ r^4+ |y|}+ \frac{1}{t} A_0  |y|^2).
\]  

\begin{rmk}\label{regularisationisnecessary}
This expression is discontinuous for $y=0$, namely on the nodal fibre, due to the non-differentiability of $|y|$ with respect to $y$. We shall deal with this problem later in Section \ref{Regularisingthesemiflatmetric} by regularisation of the metric.
\end{rmk}

Now we perform the coordinate change 
\begin{equation}\label{embedC3toX}
\mathfrak{z}_i= (\frac{t}{2A_0})^{1/3} z_i, \quad r=(\frac{t}{2A_0})^{1/6} R, \quad y=(\frac{t}{2A_0})^{2/3}\tilde{y},
\end{equation}
so that
\[
 \sqrt{ r^4+ |y|}+ \frac{1}{t} A_0  |y|^2= (\frac{t}{2A_0})^{1/3} \{ 
\sqrt{ R^4+ |\tilde{y}| }+ \frac{1}{2} |\tilde{y}|^2
\}  \sim  (\frac{t}{2A_0})^{1/3}\phi_\infty,
\]
where in the last step we are viewing $\phi_\infty$ as a regularised version of the non-smooth expression $\sqrt{ R^4+ |\tilde{y}| }+ \frac{1}{2} |\tilde{y}|^2$. We see that when we simultaneously scale the coordinates and the metric, then the leading asymptote of $\phi_{\C^3}$ in some sense matches up with the local behaviour of the semi-Ricci-flat metric.

More formally, we can view (\ref{embedC3toX}) as defining an explicit embedding map of a large open Euclidean ball $\{ |z|\lesssim t^{-1/3}   \} \subset  \C^3$ complex isomorphically onto $U_1\cap \{ |y|<\epsilon_1  \}\subset X$:
\begin{equation}
F_t: F_t^{-1}( U_1\cap \{ |y|< \epsilon_1  \}   )\subset \C^3 \to U_1\cap \{ |y|< \epsilon_1  \}\subset X.
\end{equation}
The expected behaviour is that the scaled model metric $ (\frac{t}{2A_0})^{1/3}  \omega_{\C^3}$ describes the Calabi-Yau metric $\tilde{\omega}_t$ on $U_1\cap \{ |y|<\epsilon_1  \}$ up to small error. Notice due to the prescriptions on scaling behaviours, the Euclidean ball  $\{ |z|<1  \}\subset  \C^3$ would correspond to a region in $X$ of length scale $\sim t^{1/6}$, which is the `quantisation scale' we referred to in the introduction.

\subsection{Weighted H\"older spaces on $X$}\label{WeightedHolderspacesonX}

In section \ref{Regularisingthesemiflatmetric} we shall construct an approximate CY metric $\omega_t$ on $X$, and estimate the error of its volume form. Since the actual construction is rather complicated, it is helpful to keep in mind the following rather crude picture:

\begin{itemize}
\item In the region $U_1\cap \{ |y|< \epsilon_1  \}\simeq F_t^{-1}(U_1\cap \{ |y|< \epsilon_1  \}   )$, the metric $\omega_t$ is approximately $(\frac{t}{2A_0})^{1/3}  \omega_{\C^3}$. 
\\
\item For $|y|< \epsilon_1$, but staying suitably away from the vanishing cycles in the $X_y$ fibres, the region can be identified via the diffeomorphism $G_{0}$   with a subset of the product space $X_0\times \{ |y|< \epsilon_1   \}$ (\cf (\ref{diffeomorphismG0})), and
the metric $\omega_t$ is approximately the product metric $\omega_{SRF}|_{X_0}+ \frac{ A_0 }{t}\sqrt{-1}dy\wedge d\bar{y} $.
\\
\item For $|y|>\frac{1}{2}\epsilon_1$, the metric $\omega_t$ is essentially the semi-Ricci-flat metric $\omega_{SRF}$. Since we are staying away from singular fibres $\omega_{SRF}$ is uniformly equivalent to $\omega_X+ \frac{1}{t}\tilde{\omega}_Y$.
\end{itemize}
We comment that on the overlap of the first two regions the common behaviour is described by $(\frac{t}{2A_0})^{1/3}\sqrt{-1}\partial \bar{\partial} \phi_\infty$. Similarly, there is some transition behaviour between the first two regions and the third region.

 The purpose of this section is to introduce the weighted H\"older spaces on $X$, adapted to these local geometries.

We first set up the weighted H\"older spaces $C^{k,\alpha}_{\delta, \tau}(X_0\times \C)$ on $X_0\times \C$ equipped with the product metric $\omega_{SRF}|_{X_0}+ \frac{1}{t} A_0 \sqrt{-1}dy\wedge d\bar{y}$. It is convenient to substitute the variable $\zeta= (\frac{t}{2A_0})^{-1/2}y$, so the metric becomes $\omega_{SRF}|_{X_0}+ \frac{1}{2}\sqrt{-1} d\zeta\wedge d\bar{\zeta}$. Recall on $X_0$ we have a function $r$, uniformly equivalent to the distance to the node. Now place the origin at $\zeta=0$ on the nodal line of $X_0\times \C$. Define $\rho'=\sqrt{r^2+ |\zeta|^2}$, and 
\[
w'=\begin{cases}
1 \quad \quad   &\text{if } r> \kappa \rho', \\
\frac{r}{ \kappa \rho' } &\text{if } r\leq \kappa \rho'. 
\end{cases}
\]
We define
the weighted H\"older norm on $X_0\times \C$ by
\begin{equation}
\norm{f}_{C^{k,\alpha}_{\delta, \tau }(X_0\times \C) }=  \sum_{j=0}^k \sup \rho'^{-\delta+j} w'^{-\tau+j} |\nabla^j f| +[\rho'^{-\delta+k}  w'^{-\tau+k} \nabla^k f ]_{0,\alpha}.
\end{equation}
where for any tensor $T$,
\[
[T]_{0,\alpha}=\sup_{d(x,x')\ll r(x) } \rho'(x)^\alpha  w'(x)^\alpha \frac{|T(x)-T(x')|}{d(x,x')^\alpha  }
\]
These weighted norms are adapted to viewing $X_0\times \C$ as having a local conical singularity at the origin with singular link, and are designed to resemble the weighted H\"older spaces for $(\C^3, \omega_{\C^3})$.

Define the set $U_2= \{ |y|< \epsilon_1,  r> \Lambda_1 |y|^{1/4}, r>  t^{1/6}    \}\subset X$, which can be identified via the diffeomorphism $G_0$ with an open subset $G_0(U_2)\subset X_0\times \C$. This allows one to compute the weighted H\"older norm on $U_2\simeq G_0(U_2)$. Similarly one can compute the weighted H\"older norm on $U_1$  by viewing it as $F_t^{-1}(U_1)\subset \C^3 $,  using the metric $\omega_{\C^3} $. Let $U_3=\{ x\in X: |y|>  \frac{\epsilon_1}{2}  \} \subset X$ be the subset of $X$ staying away from all singular fibres.  On $U_3$ it makes sense to compute the usual $C^{k,\alpha}$ norm using the metric $\omega_X + \frac{1}{t} {\omega}_Y$.

Now we can define the weighted H\"older spaces $C^{k,\alpha}_{\delta, \tau,t}(X)$. The weighted norm is
\begin{equation}\label{weightedHoldernorm}
\norm{f}_{C^{k,\alpha}_{\delta,\tau,t}(X)} =t^{-\frac{\delta}{6}} \norm{f}_{ C^{k,\alpha}_{\delta,\tau}(U_1, \omega_{\C^3})    }+  \norm{f}_{ C^{k,\alpha}_{\delta,\tau}(U_2) }
	+ t^{\frac{1}{2}(\delta-\tau)   } \norm{f}_{ C^{k,\alpha}(U_3) } 
\end{equation}
Similarly, one can define the $C^{k}_{\delta, \tau,t}$ norm, namely by setting $\alpha$ to zero. The definitions also extend to tensors, with a sutble twist to the powers of $t$ to maintain compatibility with differentiation. For instance, for a 2-form $\theta$
\[
\norm{\theta }_{C^{k,\alpha}_{\delta-2,\tau-2,t}(X)} =t^{-\frac{\delta}{6}} \norm{\theta}_{ C^{k,\alpha}_{\delta-2,\tau-2}(U_1, \omega_{\C^3})    }+  \norm{\theta}_{ C^{k,\alpha}_{\delta-2,\tau-2}(U_2) }
+ t^{\frac{1}{2}(\delta-\tau)   } \norm{\theta}_{ C^{k,\alpha}(U_3) }.
\]

To see (\ref{weightedHoldernorm}) is a reasonable definition, we can check that on the mutual overlap of $U_1$, $U_2$ and $U_3$, the different definition of norms are equivalent up to a bounded factor independent of $t$.  On $U_1\cap U_2$, the metric $t^{1/3}\omega_{\C^3}\sim t^{1/3}\sqrt{-1} \partial \bar{\partial} \phi_\infty$ is uniformly equivalent to the metric $G_0^*(\omega_{SRF}|_{X_0}+ \frac{1}{t}A_0 \sqrt{-1}dy\wedge d\bar{y}   ) $. % the fibre direction follows from work in section \ref{CYmetricsonsmoothingsofthenodalK3fibre}, and the horizontal direction is dominated by the term $\frac{1}{t}A_0 \sqrt{-1}dy\wedge d\bar{y}$. 
The weight functions are related on $U_1\cap U_2$, up to bounded factors, by
\[
\rho'\sim t^{1/6} \rho, \quad r\sim t^{1/6} R, \quad w'\sim w.
\] 
This is enough to conclude the equivalence of $t^{-\delta/6}\norm{\cdot}_{C^1_{\delta,\tau}(U_1, \omega_{\C^3}) }$ with $\norm{\cdot}_{C^1_{\delta,\tau}(U_2)} $ on $U_1\cap U_2$. The higher order equivalence is similar. Likewise with $U_1\cap U_3$ and $U_2\cap U_3$.

\begin{rmk}\label{weightedHoldernormcomputationtechnique}
If we focus on a normal neighbourhood region close to a given fibre $X_{y'}$ with $|y'|\gtrsim t^{1/2}$ (so that $\rho$ is predominantly $|\tilde{y}|$), we can take a nice trivialisation around $X_{y'}$, use the product metric $\omega_{SRF}|_{X_{y'}}+ \frac{1}{t} A_{y'} \sqrt{-1}dy\wedge d\bar{y}$ on $X_{y'}\times \C$ to measure the magnitudes of higher derivatives, and then turn on suitable weights
$\rho'^\delta w'^\tau\sim \rho'^{\delta-\tau} r^\tau \sim (t^{-1/2}|y|)^{\delta-\tau} r^\tau  $. This would give an equivalent definition of the weighted H\"older norm in this region up to a bounded factor independent of $t$.
\end{rmk}

\subsection{Regularising the semi-Ricci-flat metric}\label{Regularisingthesemiflatmetric}

The aim of this section is to produce an approximate CY metric $\omega_t$ on $X$. The heuristic idea, as explained in section \ref{GeometryofthemodelmetricC3}, is to glue a scaled copy of $\omega_{\C^3}$ to the semi-Ricci-flat metric $\omega_{SRF}$. This is complicated by the need to regularise $\omega_{SRF}$, pointed out in Remark \ref{regularisationisnecessary}. The rough idea of this regularisation is to replace $\omega_{SRF}$ by local product metrics when we are far from the vanishing cycles, and utilise the construction of the model metric on $\C^3$ when we are close to the vanishing cycles.

To save writing, we will pretend there is only one nodal fibre for $\pi$, although the presence of many nodal fibres causes no extra difficulty. We define a partition of unity $\{\chi_i\}_{i=0}^N$ on the base $Y$, such that 
$\chi_0=1$ on $\{ |y|\leq t^{ \frac{6}{14+\tau } }  \}$ and the support of $\chi_0$ is contained in $\{ |y|\leq 2t^{ \frac{6}{14+\tau } }  \}$. For $1\leq i\leq N$,
the supports of $\chi_i$ are contained in the complement of $\{ |y|\geq t^{ \frac{6}{14+\tau } }  \}$ for some fixed number $-2<\tau<0$, each having length scale $t^{1/2}$ in the $\omega_Y$ metric, containing a point $y_i$ which we think of as the centre of that support.   We can demand that $0\leq \chi_i\leq 1$, and all these $\chi_i$ have uniform $C^k$ bounds with respect to the metric $\frac{1}{t}\omega_Y$ for any given positive integer $k$. Morever, at each point in $Y$ the number of non-vanishing $\chi_i$ is bounded independent of $t$, even though $N\sim O(\frac{1}{t} )$.

We can now write down the metric ansatz $\omega_t$ as 
\begin{equation}\label{omegat}
\begin{split}
&\omega_t= \omega_X+ \frac{1}{t} \tilde{\omega}_Y + \sqrt{-1}\partial \bar{\partial}
\{
\sum_{i=1}^{N} \chi_i G_{y_i}^* \psi_{y_i} + 
 \chi_0(c_0+ \\
& \gamma_1(  \frac{r}{ t^{1/10}+ t^{\frac{1}{12} }\rho'^{1/6}  }     ) G_0^* (\psi_0-c_0)  + \gamma_2( \frac{r}{ t^{1/10}+ t^{\frac{1}{12}} \rho'^{1/6}  }   ) (\frac{t}{2A_0})^{1/3} ( \phi_{\C^3}' + \sqrt{R^4+ \rho  } )   )
\}
\end{split}
\end{equation}

%??????define $G_y=G_{y,y'}$ ?????

\begin{rmk}\label{omegatmeaning}
We explain the meaning of this construction, in the order of decreasing length scales, before carrying out the error estimates. The fact that $\omega_t$ is indeed a K\"ahler metric, namely it is positive definite, will be clear in the course of these estimates. As a caveat $\omega_t$ is not smooth, due to the non-smoothness of $\tilde{\omega}_Y$ (\cf section \ref{ThegeneralisedKahlerEinsteinmetric}).

\begin{itemize}
\item When $|y|\geq 2t^{\frac{6}{14+\tau} }$, including in particular $|y|\geq \epsilon_1$, we are far from the singular fibre, and the construction is $\omega_t= \omega_X+ \frac{1}{t} \tilde{\omega}_Y+ \sqrt{-1}\partial \bar{\partial}
\sum_{i=1}^{N} \chi_i G_{y_i}^* \psi_{y_i}$. We recall from (\ref{semiflatmetric1}) that $\psi_{y_i}$ is the potential of the Calabi-Yau metric on $X_{y_i}$, which we can graft to its nearby fibres using the diffeomorphism $G_{y_i}$ (here $G_{y_i}$ is well defined over the support of $\chi_i$, and only a small number of $\chi_i$ actually contribute around a given fibre $X_y$). The resulting $\omega_t$ is very close to the semi-Ricci-flat metric. Remark \ref{Scaleofdeformation} explains the special choice of power $t^{\frac{6}{14+\tau} }$. 
\\

\item When $t^{\frac{6}{14+\tau   } }\leq |y|\leq 2t^{  \frac{6}{14+\tau   }      }$, the metric $\omega_t$ starts to receive contribution from the nodal fibre $X_0$ (here the diffeomorphism  $G_0$ is well defined on the support of the cutoff functions and is used to graft the potential on $X_0$ to $X_y$), but the fluctuation effect of $\omega_{\C^3}$ is not yet significant. The expression inside $\chi_0$ plays the same role as the potential of the approximate metric $\omega_y$ on $X_y$ as in  (\ref{CYmetriconXyapproximate}).
\\

\item When $|y|< t^{ \frac{6}{14+\tau}  }$ but $r> (t^{1/10}+ t^{1/12}  \rho'^{1/6})$, the metric $\omega_t$ is essentially
\[
\omega_t\sim \omega_X+ \frac{1}{t}\tilde{\omega}_Y + \sqrt{-1} \partial \bar{\partial} G_0^* \psi_0,
\]
which is approximately the product metric on $U_2\subset X_0 \times \C$. We now summarize the basic numerical properties of the cutoff scales.  For $|y| \gtrsim t^{3/5  } $, namely $|\tilde y|\gtrsim t^{-1/15}$, the term $t^{ 1/12 }\rho'^{1/6}\geq |y|^{1/6}$ dominates the term $t^{1/10}$, so the cutoff scale of $\gamma_1(\frac{r}{t^{1/10}+ t^{1/12}\rho'^{1/6}  } )$ is comparable to the cutoff scale of $\gamma_1( \frac{r}{|y|^{1/6}  }    )$ in agreement with the gluing scale for $\omega_y$ in (\ref{CYmetriconXyapproximate}), explaining our choices of exponents in the cutoff functions.
For $|y|< t^{3/5}$, the cutoff scale of $\gamma_1( \frac{r}{t^{1/10}+ t^{1/12}\rho'^{1/6}  }    )$ is comparable to the cutoff scale of $\gamma_1( \frac{r}{t^{1/10}  }    )$, deviating from the gluing scale of $\omega_y$. The transition between these two behaviours happens at $|y|\sim t^{3/5}$, for which the cutoff scale is $r\sim t^{1/10}, \rho'\sim t^{1/10} $, and in terms of the coordinates on $\C^3$ this means $\rho\sim R\sim t^{-1/15}$. The fact that $\rho$ is comparable to $R$ indicates that the effect of regularisation on the semi-Ricci-flat metric becomes appreciable.
\\

\item
When $|y|< t^{\frac{6}{14+\tau }   }$, and $r< t^{1/10}+ t^{1/12}\rho'^{1/6}$, the metric is
\[
\omega_t= \omega_X+ \frac{1}{t} \tilde{\omega}_Y + \sqrt{-1}\partial \bar{\partial} (\frac{t}{2A_0})^{1/3} (\phi_{\C^3}'+ \sqrt{R^4+\rho}).
\]
We remark that this region is contained in $U_1$, so we can freely use the coordinates on $F_t^{-1}(U_1)\subset \C^3$.
If we replace $\frac{1}{t}\tilde{\omega}_Y$ by its leading term 
\[
\frac{1}{t} A_0\sqrt{-1} dy\wedge d\bar{y}= (\frac{t}{2A_0} )^{1/3} \frac{\sqrt{-1}}{2}d\tilde{y} \wedge d\bar{\tilde{y}}=  (\frac{t}{2A_0} )^{1/3} \sqrt{-1} \partial \bar{\partial}  (\frac{1}{2}|\tilde{y}|^2),
\]
then we can recognise that 
\[
\omega_t\sim \omega_X + (\frac{t}{2A_0})^{1/3}\sqrt{-1}\partial \bar{\partial}\{
\sqrt{R^4+ \rho}+ \frac{1}{2}|\tilde{y}|^2+ \phi_{\C^3}'\}= \omega_X+ (\frac{t}{2A_0})^{1/3}\omega_{\C^3}.
\]
But $\omega_X$ is in fact far smaller than $(\frac{t}{2A_0})^{1/3}\omega_{\C^3}$, so we are left with $\omega_t\sim (\frac{t}{2A_0} )^{1/3} \omega_{\C^3}$. As explained in section \ref{GeometryofthemodelmetricC3}, this region contains the subset  $ \{ r\lesssim t^{1/6}, R\lesssim 1, |\tilde{y}|\lesssim 1, |y|\lesssim t^{2/3} \}$ at the `quantisation scale', where the semi-Ricci-flat approximation breaks down completely. 
\end{itemize}
\end{rmk}

We now turn to the error estimates, and start with the regions where the semi-Ricci-flat behaviour is dominant. We first calculate how much the metric $\omega_t$ restricted in the fibre direction deviates from the Calabi-Yau metric on the fibres. This is a familiar problem given the work in section \ref{CYmetriconXy}, so we will only indicate main modifications.

\begin{lem}
Fix $-2<\tau<0$.
When $|y|\geq t^{ \frac{6}{14+\tau}  }$, the deviation of $\omega_t|_{X_y}$ from the CY metric $\omega_{SRF}|_{X_y}$ is estimated by
\begin{equation}\label{omegatfibrewiseerror1}
\norm{\omega_t|_{X_y}- \omega_{SRF}|_{X_y} }_{ C^{k, \alpha}_{\tau-2}(X_y)  } \leq C(k,\alpha, \tau)t^{1/2} |y|^{-\frac{1}{4}(\tau+2)   }.
\end{equation}
When $|y|< t^{\frac{6}{14+\tau}   }$, but $r\gtrsim t^{1/10}+ t^{1/12}\rho'^{1/6}$, the deviation of $\omega_t|_{X_y}$ from $G^*_{0,y}(\omega_{SRF}|_{X_0})$ is estimated by
\begin{equation}\label{omegatfibrewiseerror2}
\norm{ \omega_t|_{X_y}- G_{0,y}^*(\omega_{SRF}|_{X_0} )}_{C^{k,\alpha}_{-4}(X_y\cap \{ r\gtrsim t^{1/10}+ t^{1/12} \rho'^{1/6}   \}  )  } \leq C(k, \alpha, \epsilon, \tau)
\max{ ( |y|, t^{3/5-\epsilon } )  }
\end{equation}
where $\epsilon>0$ can be made arbitrarily small. 
\end{lem}

\begin{proof}
When $|y|\geq 2 t^{\frac{6}{14+\tau }  }$, since the support of $\chi_i$ has $\omega_Y$-length scale $\sim t^{1/2}$, we use Lemma \ref{Lipschitztypeestimatepotentialfunction} and the ensuing Remark to see that $\norm{\psi_y- G_{y_i, y}^*\psi_{y_i} }_{C^{k+2, \alpha}_{\tau}(X_y)    } \leq Ct^{1/2}|y|^{-\frac{1}{4}(\tau+2)  } $ whenever $\chi_i\neq 0$ at $y$.
Since at any $y$ the number of non-vanishing $\chi_i$ is bounded independent of $t$, these errors cannot accumulate, so adding up $\chi_i(  \psi_y- G_{y_i, y}^*\psi_{y_i}     )$ and applying $\sqrt{-1}\partial \bar{\partial}$ in the fibre direction, we see (\ref{omegatfibrewiseerror1}).
%\[
%\norm{\omega_t|_{X_y}- \omega_{SRF}|_{X_y} }_{ C^{k, \alpha}_{\tau-2}(X_y)  } \leq Ct^{1/2} |y|^{-\frac{1}{4}(\tau+2)   }%\leq Ct^{-\frac{1}{2}(\tau+1) }\rho'^{-\frac{1}{4}(\tau+2) }.
%\] 
%The last inequality uses $|y|\sim t^{1/2}\rho'$ in this region.

When $|y|\sim  t^{\frac{6}{14+\tau }      }$, we can make a few simplifications to (\ref{omegat}) with negligible effects. The cutoff function $\gamma_1( \frac{r}{t^{1/10}+ t^{1/12} \rho'^{1/6}   }    )$ is practically replaceable by $\gamma_1(\frac{r}{|y|^{1/6}}    )$, and likewise with $\gamma_2$.
We can also replace $\sqrt{R^4+\rho}$ with $\sqrt{R^4+ |\tilde{y}|}$, and use the estimate (\ref{phiC3error}) to drop the $\phi_{\C^3}'$ term in (\ref{omegat}). Then the potential term in (\ref{omegat}) proportional to $\chi_0$ is reduced to
\[
c_0+ \gamma_1(\frac{r}{ |y|^{1/6} } )G_{0,y}^*(\psi_0-c_0)+ \gamma_2(\frac{r}{|y|^{1/6}}  )\sqrt{r^4+|y|},
\]
which by Proposition \ref{CYmetriconXy} and its ensuing Remark, deviates from $\psi_y$ by $\psi_y'+c_0'(y)$, with estimate
\[
\norm{\psi_y'+c_0'(y)}_{C^{k,\alpha}_\tau(X_y)  }\leq C|y|^{-\frac{1}{6}\tau+ \frac{2}{3}  }\sim Ct^{1/2}|y|^{-\frac{1}{4}(\tau+2)  }.
\]
This contribution is comparable in strength to $\norm{\psi_y-G_{y_i,y}^*\psi_{y_i} }_{ C^{k+2, \alpha}_\tau(X_y)  }$, so we have (\ref{omegatfibrewiseerror1})
%\[
%\norm{\omega_t|_{X_y}- \omega_{SRF}|_{X_y}    }_{C^{k,\alpha}_{\tau-2}(X_y) }\leq Ct^{1/2}|y|^{-\frac{1}{4}(\tau+2)  }
%\]
as in the previous case.

When $|y|< t^{\frac{6}{14+\tau}   }$, but $r> 2(t^{1/10}+ t^{1/12} \rho'^{1/6}   )$, we have $\omega_t=\omega_X+ \frac{1}{t}\tilde{\omega}_Y+\sqrt{-1}\partial \bar{\partial} G_0^*\psi_0$. Restricted to the fibres, this situation is identical with what we saw in Lemma 
\ref{CYmetriconXyintermediatelemma}, and
\[
\norm{ \omega_t|_{X_y}- G_{0,y}^*(\omega_{SRF}|_{X_0} )}_{C^{k,\alpha}_{-4}(X_y\cap \{ r>2( t^{1/10}+ t^{1/12}\rho'^{1/6}    )   \}  )  } \leq C|y|.
\]

When $|y|< t^{\frac{6}{14+\tau}    }$, and $r\sim t^{1/10}+ t^{1/12} \rho'^{1/6}$, we have contributions from the cutoff region. As mentioned in Remark \ref{omegatmeaning} there are two subcases. When $|y|\gtrsim t^{3/5}$, the cutoff function $\gamma_1(\frac{r}{t^{1/10}+ t^{1/12} \rho'^{1/6}  }  )$ can be practically replaced by $\gamma_1( \frac{r}{|y|^{1/6} } )$, and likewise with $\gamma_2$. We are in a situation similar to Lemma \ref{CYmetriconXyintermediatelemma}, and the main correction term is $t^{1/3}\phi_{\C^3}'$, which by (\ref{phiC3error}) is of order $O(t^{1/3}\rho^{-1+\epsilon})$. (In fact there is another error term caused by the deviation of $t^{1/3}\sqrt{R^4+\rho}$ from $t^{1/3}\sqrt{R^4+|\tilde{y}|}$, which has to do with regularisation. This error is of order $O( t^{1/3}\rho^{-1})$, which is a little less significant than $t^{1/3}\phi_{\C^3}'$.) The correction effect of $t^{1/3}\phi_{\C^3}'$ to the metric is of order $O( t^{1/3} \rho^{-1+\epsilon} r^{-2} )$. The relative strength of this new error source compared to the error already present in the previous case, is or order
\[
O(\frac{t^{1/3}\rho^{-1+\epsilon}r^{-2}}{ |y|r^{-4}  })=O( \frac{t^{\frac{1}{2}- \frac{1}{6}\epsilon } r^2}{ \rho'^{1-\epsilon} |y| }    )=O( \frac{t^{\frac{1}{2}- \frac{1}{6}\epsilon } }{ \rho'^{1-\epsilon} |y|^{2/3} }  )=O( \frac{t^{ 1-\frac{2}{3}\epsilon } }{|y|^{5/3-\epsilon}  }  ).
\]
When we come near $|y|\sim t^{3/5}$, this new error source $t^{1/3}\phi_{\C^3}'$ overwhelms by a relative factor $O(t^{- \frac{1}{15}\epsilon})=O(t^{-\epsilon})$, while for $|y|\gg t^{3/5}$, this new error is not significant. Thus in this new region
$\{ t^{3/5}\lesssim |y|< t^{\frac{6}{14+\tau} } , r\sim t^{1/10}+ t^{1/12}\rho'^{1/6} \}$,
the previous estimate is changed to
\[
\norm{ \omega_t|_{X_y}- G_{0,y}^*(\omega_{SRF}|_{X_0} )}_{C^{k,\alpha}_{-4}(X_y\cap \{ r\sim t^{1/12}\rho'^{1/6}   \}  )  } \leq C \max{ (|y|, t^{3/5-\epsilon }  )   }.
\]

On the other hand, if $|y|< t^{3/5}$, then $\gamma_1(\frac{r}{t^{1/10}+ t^{1/12} \rho'^{1/6}  }  )$ can be practically replaced by $\gamma_1( \frac{r}{t^{1/10}})$, and likewise with $\gamma_2$. The main errors are caused by the variation of complex structures, the deviation of the nodal K3 metric from the flat orbifold metric on $\C^2/\Z_2$, and the presence of $t^{1/3} \phi_{\C^3}'$. Since we are working at the scale $r\sim t^{1/10}$, the various sources of error for the potential are of order
\[
O(\frac{|y|}{r^2} )=O(t^{2/5}), \quad O(r^4)=O(t^{2/5}), \quad O(t^{1/3}\rho^{-1+\epsilon})=O(t^{ \frac{2}{5}-\epsilon } ).
\]
The error for the metric comes at order $O(\frac{t^{\frac{2}{5}-\epsilon  }  }{r^2} )=O(t^{1/5-\epsilon} )$.
The higher order derivative estimate involves no extra difficulty. From this we see that when $|y|<t^{3/5}$ and $r\sim t^{1/10}$, 
\[
\norm{ \omega_t|_{X_y}- G_{0,y}^*(\omega_{SRF}|_{X_0} )}_{C^{k,\alpha}_{-4}(X_y\cap \{ r\sim t^{1/10}   \}  )  } \leq C t^{ \frac{3}{5}-\epsilon   }.
\]
A more uniform way to present these estimate is that for $|y|< t^{\frac{6}{14+\tau} }$, and $r\gtrsim t^{1/10}+t^{1/12}\rho'^{1/6}$, there is the estimate (\ref{omegatfibrewiseerror2}) where the exponent $\epsilon>0$ can be made arbitrarily small.
\end{proof}

Staying still in this region, we wish to estimate how much the volume form of $\omega_t$ fails to be Calabi-Yau. The defining condition of the  Calabi-Yau metric $\tilde{\omega}_t$ is
\[
\tilde{\omega}_t^3= a_t \sqrt{-1} \Omega\wedge \overline{\Omega}, \quad a_t= \int_X (\frac{1}{t} [\omega_Y]+[\omega_X])^3= \frac{3}{t} + \int_X \omega_X^3, 
\] 
where we used the normalisation $\int \sqrt{-1}\Omega\wedge \overline{\Omega}=1$, $\int_Y [\omega_Y]=1$, $\int_{X_y} \omega_X^2=1$.
Writing 
\begin{equation}
\omega_t^3= a_t (1+f_t) \sqrt{-1} \Omega\wedge \overline{\Omega},
\end{equation}
the task is to estimate the error $f_t$ in the weighted H\"older norm introduced in Section \ref{WeightedHolderspacesonX}.

\begin{lem}\label{omegaterror1}
Let	$-2<\tau<0$ and $\delta>\frac{3}{4}  \tau -\frac{1}{2}$, $\delta< \frac{2}{3}+ \frac{5\tau}{6}$, then in the region $\{ |y|\gtrsim t^{\frac{6}{14+\tau}  }  \}$ and the region $\{ |y|<  t^{\frac{6}{14+\tau}  },  r\gtrsim t^{1/10}+ t^{1/12} \rho'^{1/6} \}$, 
we have the volume error estimate
\begin{equation*}
\norm{f_t} _{ C^{0,\alpha}_{\delta-2, \tau-2, t } }
\leq 
C(\alpha, \delta, \tau)   t^{ \delta' }
,
\end{equation*}
where we denote $\delta'=-\frac{1}{2}\tau +\frac{1}{2}\delta   + \frac{6}{14+\tau}   (\frac{2}{3} + \frac{5\tau}{6}-\delta )$.
\end{lem}

\begin{proof}
The dominant term of $\omega_t^3$ is $\frac{3}{t}\omega_{SRF}|_{X_y}^2 \tilde{\omega}_Y= \frac{3}{t}\sqrt{-1}\Omega\wedge \overline{\Omega}$. The deviation comes from two sources: the fibrewise deviation of $\omega_t^2|_{X_y}$ from $\omega_{SRF}^2|_{X_y}$, and also $(\omega_t- \frac{1}{t} \tilde{\omega}_Y)^3$, which involves understanding the horizontal component of $\omega_t-	\frac{1}{t} \tilde{\omega}_Y$ and can be thought as fluctuation of the generalised KE metric $\tilde{\omega}_Y$.

Consider first the region with $|y|\geq 2t^{ \frac{6}{14+\tau}  }$. Fibrewise deviation from Calabi-Yau metric causes an error $f_t'= \frac{\omega_t|_{X_y}^2 \wedge \tilde{\omega}_Y }{ \sqrt{-1}\Omega\wedge \overline{\Omega}  }-1= \frac{\omega_t|_{X_y}^2  }{ \omega_{SRF}|_{X_y}^2  }-1 $, whose pointwise magnitude is controlled by
\[
\begin{split}
|f_t'|& \leq C\norm{ \omega_t|_{X_y}- \omega_{SRF}|_{X_y}}_{ C^{0}_{\tau-2}(X_y)  } r^{\tau-2} 
\leq Ct^{1/2}|y|^{-\frac{1}{4}(\tau+2)  }r^{\tau-2} \\
&\leq Ct^{1/2}(t^{1/2} \rho' )^{-\frac{1}{4}(\tau+2)   }   
\rho'^{  \tau-\delta }
r^{\tau-2}\rho'^{\delta-\tau} \\
&\leq 
C t^{  \frac{6}{14+\tau} (\frac{3}{4}\tau-\delta-\frac{1}{2}) } t^{\frac{1}{2}(1+\delta-\tau) }   r^{\tau-2}\rho'^{\delta-\tau} \\
&\leq C t^{  \frac{6}{14+\tau} (\frac{3}{4}\tau-\delta-\frac{1}{2})+\frac{1}{2}(1+\delta-\tau) }   w'^{\tau-2}\rho'^{\delta-2} = C t^{\delta'} w'^{\tau-2}\rho'^{\delta-2} .
\end{split}
\]
 where we used (\ref{omegatfibrewiseerror1}) and $\delta>\frac{3}{4}  \tau -\frac{1}{2} $.
This is the first step towards   estimating $f_t'$ in the 
 $C^{0,\alpha}_{\delta-2, \tau-2, t}(X \cap \{ |y|> 2t^{ \frac{6}{14+\tau} }  \} )$ norm 
in this region. Estimating the vertical derviatives of $f_t'$ poses no further difficulty.

We now make some general comments about horizontal differentiation. Near a given fibre $X_{y'}$, there is a trivialisation around a small normal neighbourhood, for example induced by the diffeomorphism $G_{y_i}$ where $|y_i-y'|\leq \epsilon_2 |y'|$. This will induce some horizontal distribution, which allows us to lift the vector fields on the base $Y$ to $X$. In the coordinate neighourhood $U_1$ with coordinates $\mathfrak{z}_1, \mathfrak{z}_2, \mathfrak{z}_3$, a particular lift of $\frac{\partial}{\partial y}$ is given by $\sum\frac{ \bar{\mathfrak{z}}_i }{ 2|\mathfrak{z}|^2 }  \frac{\partial}{\partial \mathfrak{z}_i}$, which is orthogonal to the fibres with respect to the standard Euclidean metric in these coordinates. Now if the trivialisation is chosen well, its induced horizontal lift of $\frac{\partial}{\partial y}$ will differ from $\sum\frac{ \bar{\mathfrak{z}}_i }{ 2|\mathfrak{z}|^2 }  \frac{\partial}{\partial \mathfrak{z}_i}$    by some vertical vector field whose $\omega_X$-magnitude is $O(\frac{1}{|\mathfrak{z}|}  )=O(r^{-2})$, or equivalently its magnitude with respect to the Eguchi-Hanson metric $EH_y$ is $O(r^{-3})$. This measures the deviation between horizontal lifts for any two different good choices of trivialisations, such as  $G_{y_i}$ and $G_{y_j}$ where $|y_i-y_j|\leq \epsilon_2|y_i|$.

In particular, given a function $f\in C^{1,\alpha}_{\tau}(X_{y_i})$  on a very nearby fibre $X_{y_i}$, then $G_{y_i}^*f$ defines a function near $X_{y'}$. To estimate the magnitude of its gradient in the horizontal direction, we can fix a good auxiliary trivialisation around $X_{y'}$, equip the normal neighbourhood with an ambient metric comparable to the product metric $\omega_{SRF}|_{X_{y'}} + \frac{1}{t}\sqrt{-1} A_{y'} dy\wedge d\bar{y}$, find the horizontal lift $v$ of $\sqrt{t}\frac{\partial}{\partial y}$ under the good trivialisation, make $v$ act on $G_{y_i}^*f$, and then compute the maginitude of the derivative (\cf Remark \ref{weightedHoldernormcomputationtechnique}). (The normalisation on $v$ is to make sure it is roughly of unit length in our ambient metric.)  But $G_{y_i}$ also provides a good trivialisation, hence another lift $v'$ of $\sqrt{t}\frac{\partial}{\partial y}$, with $|v-v'|_{EH_y}=O(t^{1/2}r^{-3} )$. Tautologically $v'(G^*_{y_i}f)=0$, so 
\[ |v (G^*_{y_i}f  )|= |(v-v') (G^*_{y_i}f  )| \leq Ct^{1/2} r^{-3} |\nabla_{X_{y_i}} f |_{EH_{y_i}}\leq Ct^{1/2} r^{-3} r^{\tau-1} \norm{f }_{ C^{1,\alpha}_\tau(X_{y_i}) } . \] We may think of $v(G_{y_i}^*f )$ suggestively as the horizontal derivative of $G_{y_i}^*f$, and write it schematically as $\sqrt{t}\frac{G_{y_i}^*f }{\partial y }$. Similarly we make sense of $\sqrt{t}\frac{G_{y_i}^*f }{\partial \bar{y} } $.
Continuing in a similar fashion, if we differentiate $G_{y_i}^*f$ by $k$ times and measure it using the ambient metric, then as long as our choices of trivialisations are well behaved (meaning $v-v'$ have good higher order weighted H\"older estimates), we will get
\[
t^{k/2}|\frac{\partial^k G_{y_i}^*f }{\partial^j y \partial^{k-j} \bar{ y} }|\leq C t^{k/2}r^{-3k} r^{\tau-k} \norm{ f  }_{ C^{k,\alpha}_{\tau}(X_{y_i})  }.
\]
The main effect of horizontal differentiation along a unit vector, compared to vertical differentiation, is that it brings about an extra factor of $O(t^{1/2}r^{-3} )$ for each derivative. This principle also works for tensors. The underlying reason for this principle to work is an approximate homogeneity under  $\mathfrak{z}\to \lambda \mathfrak{z}$, which reduces the problem to the case where $r\sim 1, |y'|\ll1$.

As a special observation, as long as $r\gg t^{1/6}$, horizontal differentiation is suppressed by vertical differentiation.
Using these principles, we see in particular that
\[
\norm{f_t'}_{ C^{0,\alpha}_{\delta-2, \tau-2, t  }(X\cap \{ |y|> 2t^{\frac{6}{14+\tau }   }  \}  ) }\leq C  t^{  \delta' }.
\]
But the metric $\tilde{\omega}_Y$ is only Lipschitz, so the best improvement is the $C^1_{\delta-2, \tau-2, t }$ bound.

Staying in the region $\{|y|> 2 t^{\frac{6}{14+\tau}  }  \} $,  we also need to estimate the error $f_t''= \frac{ (\omega_t- t^{-1}\tilde{\omega}_Y )^3 } {3t^{-1}\sqrt{-1} \Omega\wedge \overline{\Omega}  }$. Since $\omega_t$ is approximately $\omega_{SRF}|_{X_y}$ on the fibre, the size of $(\omega_t- t^{-1}\tilde{\omega}_Y )^3$ depends on knowing the horizontal part of $\omega_t- t^{-1}\tilde{\omega}_Y$ (the horizontal-vertical mixed terms also play a role, whose contributions can be treated similarly). This in turn requires understanding
 the horizontal second derivative of $G_{y_i}^*\psi_{y_i}$, and the horizontal component of $\sqrt{-1} \partial \bar{\partial }\{ \chi_i (G_{y_i}^*\psi_{y_i}- G_{y_j}^*\psi_{y_j}    )\}$ when the support of $\chi_i$ and $\chi_j$ overlap.
The former is estimated by $ \frac{Ct}{r^6}  $ using
 Proposition \ref{CYmetricsonsmoothingsofthenodalK3fibre} and the above principles concerning horizontal differentiation. Notice in our region this error is  insignificant compared to $f_t'$:
\[
\frac{Ct}{r^6}\ll Ct^{1/2} |y|^{\frac{-1}{4}(\tau+2)    }r^{\tau-2}.
\]  
The new feature of the latter term $\sqrt{-1} \partial \bar{\partial }\{ \chi_i (G_{y_i}^*\psi_{y_i}- G_{y_j}^*\phi_{y_j}    )\}$ comes from differentiating $\chi_i$, which by Lemma \ref{Lipschitztypeestimatepotentialfunction} can be controlled. For instance, 
\[
| (G_{y_i}^*\psi_{y_i}- G_{y_j}^*\phi_{y_j})  \sqrt{-1}\partial \bar{\partial}\chi_i | \leq C|G_{y_i}^*\psi_{y_i}- G_{y_j}^*\phi_{y_j}| \leq C t^{1/2}|y|^{-\frac{1}{4}(\tau+2)   }  r^\tau,
\]
%and similarly the maginitude of the horizontal component of $\partial(G_{y_i}^*\psi_{y_i}- G_{y_j}^*\phi_{y_j})   \bar{\partial}\chi_i $ is controlled by $| t^{1/2} \frac{ \partial}{\partial y} (G_{y_i}^*\psi_{y_i}- G_{y_j}^*\phi_{y_j})| \leq C t^{1/2}|y|^{-\frac{1}{4}(\tau+2)   }  r^{\tau-1} (\frac{t^{1/2}}{r^3 }) $, 
which is again dominated by $C t^{1/2}|y|^{-\frac{1}{4}(\tau+2)   }  r^{\tau-2}$; the same happens for all terms involving differentiating $\chi_i$. At each given point only a bounded number of $y_i$ contribute, so the errors do not accumulate, and the horizontal part of $\omega_t- \frac{1}{t}\tilde{\omega}_Y$ is dominated by $ C t^{1/2}|y|^{-\frac{1}{4}(\tau+2)   }  r^{\tau-2}$, whence the same holds for $|f_t''|$. All these indicate that $f_t''$ is less significant compared to $f_t'$. Proceeding further,
\[
\norm{f_t''}_{ C^{0,\alpha}_{\delta-2, \tau-2, t  }(X\cap \{ |y|> 2t^{\frac{6}{14+\tau }   }  \}  ) }\leq C  t^{  \delta' }.
\]
The total error $f_t$ can be expressed as
\[
f_t= -1+ \frac{ 3  }{t a_t  } (1+ f_t'+ f_t'').
\]
Here the normalising constant is $\frac{3}{ta_t}=1+ O(t)$, so combining the above discussions, in the region $\{ |y|> 2t^{ \frac{6}{14+\tau  } } \}$,
\[
\norm{f_t} _{ C^{0,\alpha}_{\delta-2, \tau-2, t  }(X\cap \{ |y|> 2t^{\frac{6}{14+\tau }   }  \}  ) }\leq C  t^{  \delta' }.
\] 
	
In the region $\{ |y|\sim t^{\frac{6}{14+\tau } }  \}$, there are new contributions from the terms in (\ref{omegat}) inside $\chi_0$. The arguments are very similar once we have (\ref{omegatfibrewiseerror1}). The main new features to observe is that $\phi_{\C^3}'$ is negligible using (\ref{phiC3error}), and that the cutoff functions $\gamma_1$ and $\gamma_2$ have $C^{k,\alpha}_{0,0,t}$ estimates, so multiplication by such cutoff functions only increases $C^{k,\alpha}_{\delta-2, \tau-2, t}$ norms by a bounded factor. The result is 
\[
\norm{f_t} _{ C^{0,\alpha}_{\delta-2, \tau-2, t  }(X\cap \{ |y|\sim t^{\frac{6}{14+\tau }   }  \}  ) }\leq C  t^{  \delta' }.
\]

Next, we focus on the region with $|y|< t^{ \frac{6}{14+\tau}  }$, but $r\gtrsim  ( t^{1/10}+ t^{1/12}\rho'^{1/6} )$. As before we start with the contribution $f_t'$ measuring the failure of fibrewise Calabi-Yau condition,  making use of (\ref{omegatfibrewiseerror2}) and the fact that the fibrewise holomorphic volume form $\Omega_y$ is close to $\Omega_0$. In the subcase of $|y|> t^{3/5}$, since $-2<\tau<0$, $\delta< \frac{2}{3}+ \frac{5\tau}{6} $, $0<\epsilon\ll1$ and $r\gtrsim |y|^{1/6}$,
\[
\begin{split}
|f_t'| &\leq C \max{ (t^{3/5-\epsilon},|y|)} r^{-4}
\leq C    (\max{ (t^{3/5-\epsilon},|y|)}  \rho'^{\tau-\delta} r^{-2-\tau}   ) \rho'^{\delta-2}w'^{\tau-2 } \\
&\leq Ct^{ -\frac{1}{2}   (\tau-\delta) } (\max{ (t^{3/5-\epsilon} ,|y|)}   |y|^{\tau-\delta} |y|^{\frac{1}{6}  (-2-\tau)}   ) \rho'^{\delta-2}w'^{\tau-2 } \\
& \leq Ct^{  -\frac{1}{2}\tau +\frac{1}{2}\delta    }  t^{  \frac{6}{14+\tau}   (2/3 + 5\tau/6-\delta ) }    \rho'^{\delta-2}w'^{\tau-2 } \\
& =  Ct^{  \delta' }    \rho'^{\delta-2}w'^{\tau-2 }  .
\end{split}
\]
On the other hand, when $r> 2(t^{1/10}+ t^{1/12}\rho'^{1/6}  )$, the fluctuation error $f_t''$ is %comes from taking horizontal second derivative of the potential, which cause error of order $O(\frac{t}{r^6} )$, and the deviation of $\tilde{\omega}_Y$ from being Euclidean, which causes error of order $O(|y|)$.
\[
|f_t''| \leq C\frac{t}{r^6} \ll |y| r^{-4},
\] 
so is insignificant compared to $f_t'$. At the cutoff scale, there is an extra term coming from $t^{1/3}\phi_{\C^3}'$, which gives a contribution to $f_t''$ of order $O( t^{1/3}\rho^{-1+\epsilon} r^{-2}  )$, which is again dominated by $C \max{ (t^{3/5-\epsilon} ,|y|)} r^{-4}$.
Proceeding further, 
\[
\begin{split}
\norm{f_t} _{ C^{0,\alpha}_{\delta-2, \tau-2, t  }( \{ t^{3/5}\lesssim  |y|< t^{\frac{6}{14+\tau }   }, r\gtrsim  t^{1/10}+ t^{1/12}\rho'^{1/6}   \}  ) }
\leq Ct^{  \delta' }.
\end{split}
\]

%We may observe that 
%\[
%t^{  -\frac{1}{2}\tau +\frac{1}{2}\delta   + \frac{6}{14+\tau}   (2/3 + 5\tau/6-\delta ) }
%=  t^{  \frac{6}{14+\tau} (\frac{3}{4}\tau-\delta-\frac{1}{2})+\frac{1}{2}(1+\delta-\tau) },        
%\]
%meaning that the estimate matches with the previous region $\{ |y|> t^{\frac{6}{14+\tau}   } \}$. %This is not coincidental, but is a consequence of our choices of cutoff scales, which aims at balancing various sources of errors.

In the region  $ \{|y|<t^{3/5  }, r \gtrsim t^{1/10}  \}$,   the error due to failure of fibrewise Calabi-Yau condition can be estimated by (\ref{omegatfibrewiseerror2})
\[
|f_t'| \leq C t^{3/5-\epsilon} r^{-4}\leq C t^{ 2/5-\delta/10-\epsilon  } \rho'^{\delta-2}w'^{\tau-2},
\]
and the fluctuation error is
\[
|f_t''|\leq
\begin{cases}
 C \frac{t}{r^6} \ll Ct^{3/5}r^{-4}, & r> 2(t^{1/10}+ t^{1/12}\rho'^{1/6}) \\
 Ct^{1/3} \rho^{-1+\epsilon}r^{-2}\leq    Ct^{3/5-\epsilon}r^{-4}, & r\sim t^{1/10},
\end{cases}
\]
so $|f_t'|$ is the dominant error.  Proceeding as usual,
\[
\norm{f_t} _{ C^{0,\alpha}_{\delta-2, \tau-2, t  }(\{ |y|< t^{3/5   }, r\gtrsim t^{1/10} \}  ) }
\leq Ct^{ 2/5-\delta/10-\epsilon }\ll
t^{  \delta' }
.
\]
Combining all the discussions above gives the claim.
\end{proof}

We now turn our attention to the region  $\{ |y|<t^{\frac{6}{14+\tau}  }, r< t^{1/10}+ t^{1/12} \rho'^{1/6} \}$, contained in $U_1$. Recall from Remark \ref{omegatmeaning} that the metric $\omega_t$ is predominantly $(\frac{t}{2A_0})^{1/3} \omega_{\C^3}$.
 
\begin{lem}\label{omegaterror2}
Let $-2<\tau<0$ and $\delta< \frac{2}{3}+ \frac{5\tau}{6} $, then in the region $\{ |y|<t^{\frac{6}{14+\tau}  }, r< t^{1/10}+ t^{1/12} \rho'^{1/6}  \}$, we have the estimate for the volume form error
\begin{equation*}
\norm{f_t}_{  C^{0,\alpha}_{\delta-2, \tau- 2,t}  } \leq C(\alpha, \delta, \tau) t^{ \delta' },
\end{equation*}
where we recall $\delta'=\frac{6}{14+\tau}(   \frac{2}{3}+ \frac{5\tau}{6}-\delta) + \frac{1}{2}\delta-\frac{1}{2}\tau $.
\end{lem}

\begin{proof}
As explained in Remark \ref{omegatmeaning}, in this region the deviation of $\omega_t$ from $(\frac{t}{2A_0})^{1/3} \omega_{\C^3}$ arises from  $\frac{1}{t}(A_y-A_0)\tilde{\omega}_Y$ and $\omega_X$.

In the coordinates $z_1, z_2, z_3$ on $F_t^{-1}(U_1)\subset \C^3$, the metric $\omega_X$ is comparable to the Euclidean metric $t^{2/3} \sqrt{-1} \partial \bar{\partial} R^4$, from which we get a bound for the weighted H\"older norm on $\C^3$, 
\[
\norm{ \omega_X}_{ C^{k,\alpha}_{2, 2} (U_1, \omega_{\C^3})     }\leq C t^{2/3}.
\]
and in particular its magnitude $|\omega_X|_{\omega_{\C^3}} \leq Ct^{2/3}(R+1)^2\ll t^{1/3}\sim | t^{1/3} \omega_{\C^3} |_{ \omega_{\C^3}  }$
in the region
$\{ |y|<t^{\frac{6}{14+\tau}  }, r< t^{1/10}+ t^{1/12} \rho'^{1/6}  \}$.
Morever, since $\delta< \frac{2}{3} +\frac{5\tau}{6}$, $-2<\tau<0$, we can deduce from the numerical properties of the weights that
\[
|\omega_X|_{ \omega_{\C^3}}\leq
\begin{cases}
Ct^{2/3}
R^2\leq C t^{ \frac{6}{14+\tau}(   \frac{2}{3}+ \frac{5\tau}{6}-\delta) + \frac{2}{3}\delta-\frac{1}{2}\tau  } \rho^{\delta-2}w^{\tau-2}, & R>1,\\
Ct^{2/3}\ll   C t^{ \frac{6}{14+\tau}(   \frac{2}{3}+ \frac{5\tau}{6}-\delta) + \frac{2}{3}\delta-\frac{1}{2}\tau  } , & R\lesssim 1.
\end{cases}
\]
As for $\frac{1}{t}(A_y-A_0)\tilde{\omega}_Y$, since $A_y$ is Lipschitz in $y$, this term is  $O(|y|)=O(t^{\frac{6}{14+\tau}  } )$ small compared to $\frac{1}{t}\tilde{\omega}_Y$ which is essentially the horizontal part of $t^{1/3}\omega_{\C^3}$, thus
\[
|\frac{1}{t}(A_y-A_0)\tilde{\omega}_Y |_{\omega_{\C^3}  }
\leq Ct^{1/3  } |y| \ll t^{2/3} R^2 ,
\]
so this contribution is insignificant compared to $\omega_X$. From this we deduce that the function $f_t'''=\frac{\omega_t^3}{ ( \frac{t}{2A_0}  )\omega_{\C^3}^3   }-1$ satisfies the estimate
\[
|f_t'''| \leq C t^{ \frac{6}{14+\tau}(   \frac{2}{3}+ \frac{5\tau}{6}-\delta) + \frac{2}{3}\delta-\frac{1}{2}\tau -\frac{1}{3} }
\begin{cases}
 \rho^{\delta-2}w^{\tau-2}, & R>1,\\
 1, & R\lesssim 1.
\end{cases}
\]
Proceeding further,
\[
\norm{f_t'''}_{  C^{0,\alpha}_{\delta-2, \tau- 2} (\{ |y|< t^{ \frac{6}{14+\tau} }, r< t^{1/10}+t^{1/12} \rho'^{1/6}  \}, \omega_{\C^3})   } \leq C t^{ \frac{6}{14+\tau}(   \frac{2}{3}+ \frac{5\tau}{6}-\delta) + \frac{2}{3}\delta-\frac{1}{2}\tau -\frac{1}{3} }.
\]
This cannot be improved to higher orders because we do not have higher order control for $A_y$.

Recall from section \ref{GeometryofthemodelmetricC3} that $\omega_{\C^3}^3= \frac{3}{2} \prod_1^3 \sqrt{-1} dz_i d\bar{z}_i$, so
\[
(\frac{t}{2A_0}  )\omega_{\C^3}^3 = \frac{3}{2} (\frac{t}{2A_0}  )\prod_1^3 \sqrt{-1} dz_i d\bar{z}_i
=\frac{3A_0}{t} \prod\sqrt{-1} d\mathfrak{z}_i d\bar{\mathfrak{z}}_i,
\]
which we compared to 
\[
\Omega=\sqrt{A_0} d\mathfrak{z}_1 d\mathfrak{z}_2 d\mathfrak{z}_3( 1+ O(\mathfrak{z})),\quad 
\sqrt{-1}\Omega\wedge \overline{\Omega}= \prod \sqrt{-1} d\mathfrak{z}_i d\bar{\mathfrak{z}}_i (1+ O(r^2)),
\]
where $O(\mathfrak{z})$ denotes a fixed holomorphic function. We see that 
\[
(\frac{t}{2A_0}  )\omega_{\C^3}^3
=a_t \sqrt{-1} \Omega\wedge \overline{\Omega} (1+ O(t)+ O(r^2)).
\]
The $O(r^2)$ term arises from the deviation of the holomorphic volume form $\Omega$ from  $\sqrt{A_0} d\mathfrak{z}_1 d\mathfrak{z}_2 d\mathfrak{z}_3$; its strength is comparable to the error caused by $\omega_X$ which we just analysed. 
 Both this error and $f_t'''$ contribute to $f_t= \frac{\omega_t^3}{ a_t\sqrt{-1} \Omega\wedge \overline{\Omega}   } -1 $. These lead to
\[
\norm{f_t}_{  C^{0,\alpha}_{\delta-2, \tau- 2} (\{ |y|< t^{ \frac{6}{14+\tau} }, r< t^{1/10}+t^{1/12} \rho'^{1/6}  \}, \omega_{\C^3})   } \leq C t^{ \frac{6}{14+\tau}(   \frac{2}{3}+ \frac{5\tau}{6}-\delta) + \frac{2}{3}\delta-\frac{1}{2}\tau -\frac{1}{3} }.
\]
Finally, to convert this into the weighted H\"older norm $C^{0,\alpha}_{\delta-2, \tau-2, t }$ on $X$, we need to multiply by an extra factor $t^{ -\frac{\delta-2}{6}  }$ (\cf Section \ref{WeightedHolderspacesonX}). This gives the claim.
\end{proof}

Combining the above lemmas, we get

\begin{prop}
	Let $-2<\tau<0$ and $\frac{3}{4}\tau- \frac{1}{2}< \delta< \frac{2}{3}+ \frac{5\tau}{6}$.
The volume form error is globally estimated by
\begin{equation}\label{omegaterror}
\norm{f_t}_{  C^{0,\alpha}_{\delta-2, \tau- 2,t} ( X ) } \leq C t^{ \delta' },
\end{equation}
where we recall $\delta'= \frac{6}{14+\tau}(   \frac{2}{3}+ \frac{5\tau}{6}-\delta) + \frac{1}{2}\delta-\frac{1}{2}\tau $.
\end{prop}

\begin{rmk}
	It is conceivable that some variant of the metric ansatz has smaller volume form error.
\end{rmk}

\begin{rmk}\label{HolderimpliesLinfty}
Since the norm $C^{0,\alpha}_{\delta-2, \tau-2, t}$ itself depends on $t$, some explanation is needed concerning how to appreciate the strength of an estimate like (\ref{omegaterror}). An important test is that such an estimate on any function $f$ should imply that $f$ is small in $L^\infty$ norm, which is the chief indication that nonlinear effects of the Monge-Amp\`ere equation are insignificant. Notably, in the region $U_1\simeq F_t^{-1}(U_1)\subset \C^3$, for $\delta\leq \frac{1}{2}+ \frac{3}{4}\tau $ and $-2<\tau<0$,
\[
\norm{f}_{L^\infty(U_1)} \leq C
\norm{ f} _{  C^{k,\alpha}_{\delta-2, \tau- 2} ( U_1, \omega_{\C^3} ) } \leq 
C\norm{ f} _{  C^{k,\alpha}_{\delta-2, \tau- 2,t} ( X ) } t^{ \frac{\delta-2}{6}  }.
\]
The reason for the constraint on the weight is to ensure for $R>1$, there is the inequality
\[
R^{\delta-2} w^{\tau-2} \leq C R^{\delta-2} R^{-\frac{3}{4}(\tau-2)    }\leq C.
\]
Thus (\ref{omegaterror}) is only useful for gluing purposes when
\[
-2<\tau<0, \quad  \frac{3}{4}\tau- \frac{1}{2}< \delta\leq \frac{1}{2}+ \frac{3}{4}\tau, \quad
\delta'+ \frac{\delta-2}{6}>0.
 %\frac{6}{14+\tau}( \frac{2}{3}+ \frac{5\tau}{6}-\delta)+ \frac{2}{3}\delta -\frac{1}{2}\tau- \frac{1}{3}>0.
\]
The constraint $\delta< \frac{2}{3} + \frac{5}{6}\tau$ is implied by the other constraints. These constraints have solutions, for instance, if we special to $\tau= -\frac{2}{3}$, then we need $-\frac{3}{13}<\delta\leq 0$.
\end{rmk}

\subsection{Metric deviation}\label{Metricdeviation}

In the course of estimating the  volume form error we have essentially showed the closeness of $\omega_t$ to various simpler metrics in their respective regions. We now wish to state a coarser version of these estimates, valid on somewhat larger regions. This can be viewed as a quantified statement for the intuition discussed at the beginning of section \ref{WeightedHolderspacesonX}, and will be useful in section \ref{Decompositionpatchingparametrix}.

Let $\Lambda_2\gg1, \Lambda_3\gg1$ be two large numbers, and $0<\epsilon_3\leq \epsilon_1$ be a small number, all to be fixed independent of $t$. We demand that $\epsilon_3\ll (\frac{1}{ \Lambda_1\Lambda_2^2})^{12}$ is so small that the set $\{ r> \Lambda_2^{-2} (t^{1/10}+ t^{1/12} \rho'^{1/6} ) , |y|< \epsilon_3 \}$ is contained in $U_2$. For technical convenience, we impose further that ${\epsilon_3^{1/3}}{\Lambda_2^8}\ll \epsilon_3^{1/7} $.

 %We shall only keep track of the error dependence on $\epsilon_3$, neglecting any error which is suppressed by some power of $t$, tacitly assuming $t$ to be smaller than any given positive power of $\epsilon_3$.

\begin{prop}\label{metricdeviationerror}
Given $\epsilon_3, \Lambda_2, \Lambda_3$ as above, then as long as $t$ is sufficiently small, the following estimates hold.
\begin{itemize}
	\item
In the region $ \{ r\lesssim t^{1/10}+ t^{1/12} \rho'^{1/6}   \}\cap \{ |y|< \epsilon_3  \}\subset U_1 $, the metric $\omega_t$ deviates from the scaled $\C^3$ model metric by
\[
\norm{\omega_t- (\frac{ t}{2A_0}  ) ^{1/3} \omega_{\C^3}  }_{ C^{0,\alpha}_{0, 0, t}    }\leq C(\alpha) \epsilon_3^{1/7}.
\]	

\item 
In the region $\{ r> \Lambda_2^{-2} (t^{1/10}+ t^{1/12} \rho'^{1/6} ) , |y|< \epsilon_3 \} \subset U_2$, which can be identified as a subset of $X_0\times \C$ via the trivialisation $G_0$, the metric $\omega_t$ deviates from the product metric by
\[
\norm{\omega_t- G_0^*( \omega_{SRF}|_{X_0}+ \frac{1}{t} A_0 \sqrt{-1} dy\wedge d\bar{y}   )     }_{ C^{0,\alpha}_{0, 0, t}    }\leq C(\alpha) \epsilon_3^{1/7}.
\]

\item
Let $X_{y'}$ be any fibre with $|y'|> \frac{\epsilon_3}{2}$. For $|y-y'|\leq \Lambda_3 t^{1/2}\ll \epsilon_2 |y'|$, the trivialisation $G_{y'}$ is well defined. The metric $\omega_t$ deviates from the product metric in the region $\{ |y-y'|\leq \Lambda_3 t^{1/2} \}$ by
\[
\norm{\omega_t- G_{y'}^*( \omega_{SRF}|_{X_{y'}}+ \frac{1}{t} A_{y'} \sqrt{-1} dy\wedge d\bar{y}   )     }_{ C^{0,\alpha}_{0, 0, t}    }\leq C(\alpha, \epsilon_3) \Lambda_3 t^{1/2}.
\]
Here if $y'$ goes beyond the coordinate neighbourhood, then $A_{y'} \sqrt{-1}dy\wedge d\bar{y}$ should be replaced by the local Euclidean metric on the base which best approximates $\tilde{\omega}_Y$.

\end{itemize}
\end{prop}

\begin{proof}
(Sketch) For brevity, we will only indicate how to estimate the metric deviation in the $L^\infty$ norm. The weighted H\"older improvement is no more difficult, given the methods in Lemma \ref{omegaterror1} and \ref{omegaterror2}.

When $|y|< t^{ \frac{6}{14+\tau}  }$ the first two estimates are essentially extractable from the calculations in Lemma \ref{omegaterror1} and Lemma \ref{omegaterror2}, the point being that the metric deviation is bounded by some positive power of $t$, so when $t$ is sufficiently small all these terms are negligible compared to any bound independent of $t$.

Let $|y|\gtrsim t^{\frac{6}{14+\tau}  }$.
We consider first the deviation of $\omega_t$ from $(\frac{t}{2A_0}  )^{1/3}\omega_{\C^3}$. In the region $\{ r\lesssim t^{1/10}+ t^{1/12} \rho'^{1/6}   \}\cap \{   t^{ \frac{6}{14+\tau} } \leq |y|< \epsilon_3     \}$, around a given fibre $X_{y''}$ the metric deviation is primarily the deviation between $(\frac{t}{2A_0 })^{1/3} \omega_{\C^3}$ and the local product metric $ G_{y''}^*(\omega_{SRF}|_{X_{y''}}+ \frac{1}{t}\sqrt{-1}A_{y''} dy\wedge d\bar{y})$. In the base direction, the only deviation which is not suppressed by a power of $t$ comes from $\frac{1}{t}(A_{y''}-A_0) \sqrt{-1} dy\wedge d\bar{y}$, which is of order $O(|y|)=O(\epsilon_3)$. In the fibre direction, we can use Proposition \ref{CYmetriconXy} to estimate the deviation of $\omega_{SRF}|_{X_{y}}$  from the Eguchi-Hanson metric on fibres, which is of order $O(|y|^{\frac{2}{3}-\frac{1}{6}\beta + \frac{1}{4}(\beta-2)  } )= O( \epsilon_3^{\frac{1}{12}(\beta+2)} )$ for any $-2<\beta<0$. The fibrewise deviation of $(\frac{t}{2A_0 })^{1/3} \omega_{\C^3}$ from the Eguchi-Hanson metric is negligible in this region. 
In particular, these errors are all controlled by $C\epsilon_3^{1/7}$. The exponent is not optimal; anything less than $\frac{1}{6}$ will do.

Next we consider the deviation of $\omega_t$ from the product metric $G_0^*(\omega_{SRF}|_{X_0}+ \frac{1}{t} A_0 \sqrt{-1} dy\wedge d\bar{y})$, in the region
$
\{ r> \Lambda_2^{-2} (t^{1/10}+ t^{1/12} \rho'^{1/6} ) , t^{\frac{6}{14+\tau} }\lesssim |y|< \epsilon_3 \}.
$
Apart from the $O(\epsilon^{1/7})$ error in the previous case, there is an additional contribution, caused by the deviation of Eguchi-Hanson metric from $G_0^* (\omega_{SRF}|_{X_0})$, which is of order $O( \frac{|y|}{r^4}  )= O( \frac{|y|}{|y|^{2/3  } \Lambda_2^{-8 }} )= O ({\epsilon_3^{1/3}}{\Lambda_2^8}   )  
$. By our imposed assumptions, this term is also dominated by $O( \epsilon_3^{1/7})$.

The last claim of this Proposition deals with fibres bounded away from the singular fibre, and is therefore easy.
\end{proof}

\section{Inverting the Laplacian}

We first describe the harmonic analysis on various model spaces of $\omega_t$ at different scales, and then produce a parametrix of the Green operator by means of decomposition and patching, a method I learnt from G. Sz\'ekelyhidi \cite{Gabor}.

Let $\mathcal{A}^{0,\alpha}_{ \delta-2, \tau-2,t  }(X)$ be the space of $\partial \bar{\partial}$-exact (1,1) type forms completed under the $C^{0,\alpha}_{\delta-2, \tau-2,t  }$ norm on 2-forms. The trace over $\omega_t$ defines a bounded map
\[
\Tr_{\omega_t}: \mathcal{A}^{0,\alpha}_{ \delta-2, \tau-2,t  }(X) \to C^{0, \alpha}_{\delta-2, \tau-2, t}(X).
\]
The image lies in the subspace $\{\int_{X} f\omega_t^3=0\}$ by construction. %and the kernel is trivial as a standard matter of Hodge theory.
Our main result for the linear analysis is

\begin{prop}\label{Greenoperator}
Let $-2+\alpha<\delta<0$, $-2+\alpha<\tau<0$, and assume $\delta$ avoids a discrete set of values. Then there exists a right inverse $\mathcal{R}$ to $\Tr_{\omega_t}$ on the subspace of average zero functions,
\[
\mathcal{R}: \{  f\in C^{0, \alpha}_{\delta-2, \tau-2, t}(X)   : \int_X f\omega_t^3 =0 \} \to \mathcal{A}^{0,\alpha}_{\delta-2, \tau-2,t  } ,
\]
with norm bound $\norm{\mathcal{R}}\leq C(\delta, \tau, \alpha)$ independent of $t$.
\end{prop}

In this Chapter we will use the analyst's Laplacian $\Lap_{\omega_t}= 2\Tr_{\omega_t} \sqrt{-1} \partial \bar{\partial}$. The right inverse $\mathcal{R}$ can be thought schematically as $\mathcal{R}=2\sqrt{-1}\partial \bar{\partial} \Lap_{\omega_t}^{-1}$. It maps a real valued function to a real  (1,1)-form.
We emphasize that the $t$-independent bound is optimal, which is the main strength of the method.

\subsection{Harmonic analysis for $\omega_{\C^3}$ }

We need the mapping property of the weighted function space $C^{k,\alpha}_{\delta, \tau}(\C^3, \omega_{\C^3})$, introduced in Section \ref{GeometryofthemodelmetricC3}.

The following can be extracted from \cite{Gabor}, which shows how to invert the Laplacian outside a large ball. It is proved by producing an approximate Green operator. (Strictly speaking \cite{Gabor} deals with the Laplacian of an approximation of $\omega_{\C^3}$, but near spatial infinity their difference is negligible.)

\begin{lem}(\cf \cite{Gabor} Proposition 6)
Let $-2< \tau<0$ and let $\delta$ avoid a discrete set of values.
There exists a sufficiently large radius $A\gg1$ and an operator $P_A: C^{0, \alpha}_{\delta-2, \tau-2}(\C^3) \to C^{2, \alpha}_{\delta, \tau}(\C^3)$
with $\norm{P_A f}_{ C^{2, \alpha}_{\delta, \tau} } \leq C \norm{f}_{ C^{0, \alpha}_{\delta-2, \tau-2}  }$, such that $\Lap P_A f=f$ in the exterior region $\{|z|>A\}$.
\end{lem}

\begin{rmk}
The purpose for $\delta$ to avoid the discrete set of indicial roots, is to make sure the model Laplace operator on $\C^2/\Z_2 \times \C$ is invertible on a double weighted H\"older space (\cf Proposition 13 in \cite{Gabor}). For example, if $-2< \tau<0$ and $-2<\delta<0$, then this condition is automatic.
\end{rmk}

\begin{prop}\label{HolderC3}
Let $-2<\tau<0$, and  $\delta>-4$ avoids a discrete set of values. Then there exists a bounded right inverse   $P_{\C^3}: C^{0, \alpha}_{\delta-2, \tau-2}(\C^3)\to C^{2, \alpha}_{\delta, \tau}(\C^3)$ to the Laplacian.
\end{prop}

\begin{proof}
By the above lemma, it suffices to invert the Laplacian for functions $f$ with support in $\{|z|\leq A\}$. In particular $f$ satisfies $|f|\leq C\norm{f} \rho^{\delta'-2}$ for any choice of $-4<\delta'<\min\{\delta, 0\}$. But since $(\C^3, \omega_{\C^3})$ is Ricci flat with Euclidean volume growth, the function $\rho$ is uniformly equivalent to the distance to the origin outside the unit ball, and $\rho|\Lap \rho|+ |\nabla \rho|\leq C$, so we can apply Theorem 1.6 in \cite{Hein} to find a unique  function $u$ solving the Poisson equation with estimate
\[
\Lap u=f, \quad |u|\leq C\norm{f}_{C^{0, \alpha}} \rho^{\delta} \leq C\norm{f} \rho^{\delta} w^{\tau}.
\]
Since $u$ is harmonic in $\{|z|>A\}$, we can bootstrap this to a $C^{2, \alpha}_{\delta, \tau}$ estimate on $u$,
$\norm{u}_{ C^{2, \alpha}_{\delta, \tau}  } \leq C\norm{f}_{ C^{0, \alpha}_{\delta-2, \tau-2}      }$.
\end{proof}

\subsection{Harmonic analysis for $\text{K3}\times \C$}

Let $X_y$ be a K3 fibre equipped with the Calabi-Yau metric $\omega_{SRF}|_{X_y}$ in the class $[\omega_X|_{X_y}]$. We consider the case where either $X_y$ is one of the nodal fibres, or $y$ is bounded away from the set $S$ of critical values, so that $X_y$ has bounded geometry. We sketch the harmonic analysis on $X_y\times \C$ with the product metric $\omega_{SRF}|_{X_y}+ \frac{\sqrt{-1}}{2} d\zeta \wedge d\bar{\zeta}$, following the established method of G. Sz\'ekelyhidi \cite{Gabor}, T. Walpuski \cite{Walpuski} and S. Brendle \cite{Brendle}. Here $\zeta$ denotes the
 standard coordinate on $\C$. Of particular importance to us is an exponential decay property when the forcing term has fibrewise average zero and compact support, which will be exploited later to localise the Green operators.

We begin by working with the usual H\"older space $C^{k,\alpha}(X_y\times \C)$. Let  $C^{k,\alpha,ave}(X_y\times \C)$ denote the subspace of functions with average zero on $X_y$ fibres.

\begin{lem}\label{HolderXytimesC}
The Laplacian  $\Lap: C^{2, \alpha,ave}(X_y\times \C) \to C^{0, \alpha,ave}(X_y\times \C)$ is an isomorphism, with bounded inverse $P_y=\Lap^{-1}, \norm{P_y}\leq C(y)$. Morever, if the forcing term $f\in C^{0, \alpha,ave}$ is supported in $X_y\times \{|\zeta|<B\}$ for some $B>1$, then outside of $X_y\times \{|\zeta|<B+1\}$ the function $P_y f$ has exponential decay:
\begin{equation}
|P_y f|\leq C(y, \alpha) e^{-m (|\zeta|-B) }\norm{f}_{C^{0,\alpha}}, \quad |\zeta|\geq B+1.
\end{equation}
Similarly for the higher derivatives.
The constants $C(y)$ are independent of $B$, and are uniform for $y$ bounded away from $S$, or $y\in S$.
\end{lem}

\begin{proof}
	(sketch)
By standard Schauder estimate
\[
\norm{u}_{C^{2,\alpha}} \leq C \norm{\Lap u}_{C^{0, \alpha}}+ C\norm{u}_{L^\infty}.
\]
Applying Lemma 7.5 in \cite{Walpuski}, the kernel of $\Lap: C^{2, \alpha}(X_y\times \C) \to C^{0, \alpha}(X_y\times \C)$ must be constant on the $\C$ factor, so must be a global constant. It is clear from Fourier decomposition in the fibre direction that $\Lap$ restricts to a map between the subspaces of functions with fibrewise average zero, where the kernel of $\Lap$ is removed. Then a standard blow up argument shows the coercivity estimate $\norm{u}_{C^{2,\alpha}} \leq C \norm{\Lap u}_{C^{0, \alpha}}$. 

For the surjectivity claim, we can use Fourier analysis to invert $\Lap$ for smooth functions with fibrewise average zero, whose Fourier transform in the $\C$ direction has compact support. 
One can then remove the compact support assumption by an approximation argument in the weak topology, using the above
 coercivity estimate (\cf Page 23, \cite{Gabor} for a very similar argument).

For the exponential decay estimate, one considers the fibrewise $L^2$ integral $g(\zeta)=\int_{X_y} u^2(\cdot, \zeta)$, which is a function on $\C$. Since $\Lap u=0$ for $|\zeta|>B$, we have 
\[
\Lap_\C g= \int_{X_y} 2|\nabla_{\C} u|^2 + 2 u \Lap_\C u= \int_{X_y} 2|\nabla_{\C} u|^2 -  2 u \Lap_{X_y} u =\int_{X_y} 2|\nabla u|^2.
\]
Now since the fibrewise average is imposed to be zero, by the Poincar\'e inequality  $\int_{X_y} |\nabla u|^2 \geq m'^2 g$ for some $m'>0$, so we have
\[
\Lap_\C g \geq m'^2 g,\quad |\zeta|>B.
\]
The $L^\infty$ bound on $g$ is already bounded in terms of $\norm{\Lap u}_{C^\alpha}$. We now compare $g$ on  $\{ |\zeta|\geq B \}$
with a positive supersolution
\[
\tilde{g}_\epsilon= C\norm{\Lap u}_{C^\alpha}  \log (2|\zeta|/B) e^{-m'(|\zeta|-B)}+ \epsilon e^{\epsilon|\zeta|}, \quad 0<\epsilon\ll1
\]
satisfying $\Lap_\C \tilde{g}_\epsilon\leq  m'^2 \tilde{g}_\epsilon$,  to deduce $g\leq \tilde{g}_\epsilon$
in the region $\{ |\zeta|\geq B \}$ . Sending $\epsilon\to 0$ gives an exponential decay estimate on $g$ with any decay rate $0<m<m'$, and since $u$ is harmonic outside $\{|\zeta|\leq B\}$, by elliptic regularity this implies exponential decay on $u$ and all the higher derivatives.
\end{proof}

\begin{rmk}
	The physical intuition of this exponential decay is that massive particles have exponentially decaying Yukawa potentials.
\end{rmk}

Now let $X_0$ be a nodal K3 fibre.
Recall the double weighted H\"older space $C^{k,\alpha}_{\delta, \tau}(X_0\times \C)$ from section \ref{WeightedHolderspacesonX}, which is adapted to viewing $X_0\times \C$ as a space with local conical singularity at the origin with singular link. Recall 
the vertical distance to the nodal line is comparable to $r$, the distance in the base direction is $|\zeta|$, and $\rho'= \sqrt{ |\zeta|^2+ r^2}$ essentially measures the distance to the origin.

%When  $-2<\tau<0$ and $\delta>-2$,
%notice that  $C^{0, \alpha}_{\delta-2, \tau-2}(X_0\times \C)$ functions can be integrated along $X_0$ fibres, using the integrability of $r^{ \delta-2  }$ and $r^{\tau-2}$ along fibres. If morever $\min{(\tau, \delta)} > -2+ \alpha$, then $r^{ \tau-2-\alpha }$ and $r^{\delta-2-\alpha}$ are integrable on fibres, from which one can show that integrating a $C^{0,\alpha}_{\delta-2, \tau-2}(X_0\times \C)$ function along fibres will produce a function on $\C$ with a good weighted H\"older bound. 

\begin{lem}\label{HolderXytimesCpushforward}
Let $-2<\tau<0$, and $\delta>-2$, then functions $f$ in $C^{0,\alpha}_{\delta-2, \tau-2}(X_0\times \C)$ can be integrated along $X_0$ fibres, 
 \[
\pi_*{f}(\zeta)=\int_{X_0\times \{\zeta\} } f, \quad |\pi_*{f}(\zeta)| \leq C(1+|\zeta| )^{\delta-\tau} \norm{f}_{C^{0, \alpha}_{\delta-2, \tau-2}(X_0\times \C) }.
\]
If morever $\min{(\tau, \delta)} > -2+ \alpha$, then for $|\zeta-\zeta'|\leq 1$, 
\[
|\pi_*{f}(\zeta)- \pi_*{f} (\zeta')| \leq C(1+|\zeta| )^{\delta-\tau} |\zeta-\zeta'|^\alpha \norm{f}_{C^{0, \alpha}_{\delta-2, \tau-2}(X_0\times \C) }.
\]
\end{lem}

\begin{proof}
For the first claim, notice \[
|f|\leq C\rho'^{\delta-2}  w'^{\tau-2}\norm{f} = C\rho'^{\delta-\tau} r^{\tau-2} \norm{f} \leq  C\norm{f} \max{ (r^{\delta-2}, r^{\tau-2}  ) }   ,\]
so if $r^{\delta-2}, r^{\tau-2}$ are both integrable along fibres, then $|\pi_*f(\zeta)|\leq C\norm{f} $ for $|\zeta|\lesssim 1$. When $|\zeta|>1$, then $\rho'$ is uniformly equivalently to $|\zeta|$, so $|f|\leq C|\zeta|^{\delta-\tau} r^{\tau-2}$, and we obtain $|\pi_*f(\zeta)|\leq C |\zeta|^{\delta-\tau} \norm{f}$.

For the weighted H\"older statement, we use the integrability of $r^{\delta-2-\alpha}, r^{\tau-2-\alpha}$ instead.
\end{proof}

For such weights we can make sense of the subspace $C^{0,\alpha,ave}_{\delta-2,\tau-2}(X_0\times \C) \subset C^{0,\alpha}_{\delta-2,\tau-2}(X_0\times \C)$ of functions with zero average on fibres. Likewise with the subspace $C^{2, \alpha,ave}_{ \delta, \tau  }(X_0\times \C)\subset C^{2, \alpha}_{ \delta, \tau  }(X_0\times \C)$.
Observe also that $L^2$ integrals on fibres make sense for functions in $C^{2, \alpha}_{ \delta, \tau  }(X_0\times \C)$.

\begin{prop}\label{HolderXytimesCweighted2}
Let $-2+ \alpha<\tau<0$, and $-2+\alpha<\delta<0$.
The Laplacian $\Lap: C^{2, \alpha, ave}_{\delta, \tau}(X_0\times \C) \to C^{0, \alpha,ave}_{\delta-2, \tau-2}(X_0\times \C)$ is an isomorphism, with bounded inverse $P_{y=0}=\Lap^{-1}$. Morever, if the forcing term $f\in C^{0, \alpha,ave}_{\delta-2, \tau-2}$ is supported in $X_0\times \{|\zeta|<B\}$, then outside of $X_0\times \{|\zeta|<B+1 \}$ the function $P_{y=0} f$ has exponential decay:
	\begin{equation}
	|P_{y=0} f|\leq C(\alpha, \tau, \delta) e^{-m (|\zeta|-B)}\norm{f}_{C^{0,\alpha}_{\delta-2, \tau-2} }, \quad |\zeta|\geq B+1.
	\end{equation}
\end{prop}

\begin{proof}
From elliptic estimates
\[
	\norm{u}_{C^{2,\alpha}_ { \delta,\tau} } \leq C \norm{\Lap u}_{C^{0, \alpha}_{\delta-2,\tau-2}}+ C\norm{\rho'^{-\delta} w'^{-    \tau} u}_{L^\infty}.
\]
We claim the coercivity estimate $\norm{u}_{C^{2,\alpha}_{\delta,\tau} } \leq C \norm{\Lap u}_{C^{0, \alpha}_{\delta-2,\tau-2}}$. If this fails, then
we consider a blow up sequence with $\norm{\rho'^{-\delta} w'^{-\tau} u_i}_{L^\infty}=1$, $\norm{\Lap u_i}_{C^{0,\alpha}_{\delta-2, \tau-2} }\to 0$, and $|\rho'^{-\delta} w'^{-\tau} u_i(x_i)|>\frac{1}{2}$, where $x_i$ does not lie on the nodal line. The elliptic estimates provide a uniform $C^{2, \alpha}_{\delta, \tau}$ bound, and we will use Arzela-Ascoli to extract a subsequence to reach a contradiction. 
\begin{itemize}
\item If $x_i$ tends to the origin, and $\frac{r(x_i)}{\rho'(x_i)}$ is uniformly bounded positively from below, then we perform a metric scaling so that the distance from $x_i$ to the origin is normalised to 1, and scale the function $u_i$ to have value 1 at $x_i$. 
Passing to the scaled limit, we get a nontrivial harmonic function on the flat $\C^2/\Z_2\times \C$ away from the nodal line, with a double power law bound of the type $O(\rho'^\delta w'^\tau)$. (Here we abuse notation to denote by  $\rho'$ and $w'$ the corresponding quantities  on $\C^2/\Z_2\times \C$.) For our range of weights the harmonic function is in $L^1_{loc}$, so  extends over the nodal line by  standard elliptic regularity. But our weights also force decay in the $\C^2/\Z_2$ direction, so the harmonic function is trivial, contradiction. \\

\item If $r(x_i)\to 0$, and $\frac{r(x_i)}{\rho'(x_i)}$ tends to zero, then we perform a metric scaling so that the distance from $x_i$ to the nodal line is normalised to be 1, and scale the function $u_i$ to have value 1 at $x_i$. Passing to a scaled limit, we get a nontrivial harmonic function on $\C^2/\Z_2\times \C$, with a power law bound of the type $O(r^\tau)$, which implies a contradiction similar to the previous case. \\

\item If $x_i$ stays in a bounded region of $X_0\times \C$, and $r(x_i)$ is uniformly bounded positively from below, then after passing to the limit, we obtain a nontrivial harmonic function $u$ in $C^{2,\alpha,ave}_{\delta, \tau  }(X_0\times \C)$, which must extend smoothly across the nodal line because it is locally integrable.
We note that the fibrewise average of $u$ is zero. Then we can consider the fibrewise $L^2$ integral 
\[
g(\zeta)= \int_{ X_0\times \{\zeta\}} |u|^2, 
\]
satisfying $\Lap_{\C} g \geq m'^2 g$ for  $m'>0$
as in Lemma \ref{HolderXytimesC}, and compare it to the function $C+ \epsilon e^{\epsilon |\zeta| }$ in the region $\{|\zeta|\geq 1\}$. As $\zeta\to \infty$ the harmonic function $u$ is controlled by a power law $O(|\zeta|^{\delta-\tau })$, so must be dominated by $C+ \epsilon e^{\epsilon |\zeta| } $. By comparison principle $g\leq C+ \epsilon e^{\epsilon |\zeta| }$, and taking the limit $\epsilon\to 0$ shows that $g$ is bounded, hence $u$ is bounded, so we can apply Lemma \ref{HolderXytimesC} to conclude $u=0$, contradiction. \\

\item If $r(x_i)$ stays bounded below, but $\rho'(x_i)\to \infty$, then we normalise $u_i$ to have value 1 at $x_i$. Passing to the scaled limit we get a nontrivial harmonic function $u$ on $X_0\times \C$ with fibrewise average zero, and there is a bound of type $u=O(r^{\tau})$. Then $u$ must extend smoothly across the nodal line, and a similar argument as before shows $u=0$, contradiction.  
\end{itemize}

In particular $\Lap: C^{2, \alpha, ave}_{\delta,\tau}(X_0\times \C) \to C^{0, \alpha,ave}_{\delta-2,\tau-2}(X_0\times \C)$ is injective. For the surjectivity claim, we need to solve $\Lap u=f$ for a given $f\in C^{0, \alpha,ave}_{\delta-2,\tau-2}(X_0\times \C)$. For this we can use a sequence of functions $f_i \in C^{0, \alpha,ave}(X_0\times \C)$ which weakly converge to $f$ on compact subsets of the complement of the nodal line, and are uniformly bounded in $C^{0, \alpha,ave}_{\delta-2,\tau-2}$; such a sequence can be produced using cutoff functions and integration on fibres. Using Lemma \ref{HolderXytimesC} we find $u_i\in C^{2, \alpha}$ solving $\Lap u_i= f_i$. Notice for our range of weights $C^{2, \alpha}\subset C^{2, \alpha}_{\delta, \tau}$. By the coercivity estimate $u_i$ is a bounded sequence in $C^{2,\alpha,ave}_{\delta, \tau}$, from which we can extract a weak limit $u$, giving the desired solution to $\Lap u=f$.

The exponential decay argument is as in Lemma \ref{HolderXytimesC}.
\end{proof}

\begin{rmk}
If $\delta>0$, then the surjectivity statement must fail. If a smooth function $u\in C^{2, \alpha}_{\delta, \tau}(X_0\times \C)$, then the positive weight forces $u$ to vanish at the origin. 
But for a smooth forcing function $f$, in general we cannot expect to solve $\Lap u=f$ with fibrewise average zero and the vanishing condition at the origin.
For the rescue, it seems necessary to enlarge $C^{2,\alpha}_{\delta, \tau}$ by allowing for an extra smooth function locally constant near the origin.
\end{rmk}

\begin{rmk}
There are two reasons why we did not consider the function spaces $C^{k, \alpha}_{\delta, \tau }(X_0\times \C)$ for larger $k$. The first is that we do not have higher regularity estimates for the metric ansatz $\omega_t$ (\cf section \ref{Regularisingthesemiflatmetric}). The second is that the higher order derivatives $\nabla^k u= O(\rho'^{\delta-k}w^{\tau-k})$
can fail to be $L^1$ integrable along fibres when $k$ is large, which can also cause the disagreement between classical and distributional derivatives.
\end{rmk}

\subsection{Decomposition, patching and parametrix}\label{Decompositionpatchingparametrix}

Let $f$ be a function in the weighted space $C^{0,\alpha}_{\delta-2, \tau-2,t}(X)$ as in section \ref{WeightedHolderspacesonX}. Our aim is to find an approximate solution to $\Lap_{\omega_t} u=f$ subject to the condition $\int_X f\omega_t^3=0$. This ultimately leads to a proof of Propostion \ref{Greenoperator}.  For simplicity of presentation, we pretend there is only one nodal fibre in $X$, although more nodal fibres present no further difficulty. We shall assume thoughout this section that $-2+\alpha<\tau<0$ and $-2+\alpha<\delta<0$. Furthermore let $\delta$ avoid the discrete set of values appearing in Proposition \ref{HolderC3} (which is actually automatically true for this range of weights.) 

\begin{rmk}
In this section, we regard the parameters $\delta, \tau, \alpha,\epsilon_1, \epsilon_2, \Lambda_1$ as fixed once for all, and all constants are allowed to depend on them; we will introduce some additional parameters $\epsilon_3, \Lambda_2, \Lambda_3$ to be fixed in due course, and the dependence of various estimates on these parameters will be explicitly tracked down. It is also our standing assumption that $t$ is sufficiently small with respect to the choices of all other parameters.	
\end{rmk}

We will perform decomposition on $f$ both spatially and spectrally, so that on each constituent of $f$ the harmonic analysis becomes simpler.

Recall we have a standard cutoff function
\[
\gamma_1(s)= \begin{cases}
1 \quad \text{if } s>2, \\
0 \quad \text{if } s<1.
\end{cases}
\]
and $\gamma_2=1-\gamma_1$.
Let $\Lambda_2\gg1$  be a large number to be fixed independent of $t$. On the region $\{|y|< \epsilon_1\}$, we build cutoff functions
\[
\eta_1= \gamma_1( \frac{ r\Lambda_2}{ t^{1/10}+ t^{1/12} \rho'^{1/6}   }   ), \quad \eta_2= 1-\eta_1.
\]
These will be used to decompose the function near the singular fibre into regions modelled by $\C^3$ and  $X_0\times \C$.
We also need a partition of unity $\{\chi_i'\}_{i=0}^{N'}$ on the base $Y$. Here $\chi_0'$ is equal to 1 for $|y|\leq \epsilon_3/2$ and vanishes for $|y|> 2\epsilon_3/3$, where $0<\epsilon_3\leq \epsilon_1$ is some small number to be fixed independent of $t$, satisfying the constraints listed in section \ref{Metricdeviation}. For $i=1,2,\dots ,N'$, the functions $\chi_i'$ have supports contained in  $\{|y|\geq \epsilon_3/2\}$, each having
 length scale $\sim t^{1/2}$ in the $\omega_Y$ metric, and containing a point $y_i'$ which we think of as the centre of that support. We can demand $0\leq \chi_i'\leq 1$, and all these $\chi_i'$ have uniform $C^k$ bounds with respect to the metric $\frac{1}{t}\omega_Y$ for any given positive $k$. Morever, at each point in $Y$ there are only a bounded number of nonvanishing $\chi_i'$, even though $N'\sim O(\frac{1}{t})$.

We now decompose $f$ into pieces.

\begin{itemize}
\item
For $i=1, \ldots, N'$, we use the trivialisation $G_{y_i'}$ to identify $\chi_i' f$ as a compactly supported function $(G_{y_i'}^{-1})^*(\chi_i' f)$ on $X_{ y_i'}\times \C$, and use integration on fibre to decompose it into the sum of a function $f_i$ with fibrewise average zero in $X_{y_i'}\times \C$, and a function $f_i'$ on the base:
\[
  \chi_i' f= f_i+ f_i'.
\]
Clearly the supports of $f_i$ and $f_i'$ are contained in the support of $\chi_i'$. The $C^{0,\alpha}_{\delta-2, \tau-2, t  }$ norms of these functions $f_i$ and $f_i'$ are bounded in terms of the norm on $f$, because integration on fibre is a bounded operator for $-2+\alpha<\delta<0$ and $-2+\alpha<\tau<0$.  The reader is encouraged to think of functions on the base as the zeroth Fourier mode and fibrewise average zero functions as the higher Fourier modes.	\\

\item 
For $i=0$, we use the trivialisation $G_0$ to identify $\chi_0' \eta_1 f$ as a compactly supported function $(G_0^{-1})^*(\chi_0' \eta_1 f)$ on $X_0\times \C$. We calculate the integration on the fibres of $X_0\times \C$ to obtain a locally defined function on $\C$, given by
\[
f_0'(y)=\frac{ \int_{X_0\times \{y\}} (G_0^{-1})^*( \eta_1 f) \omega_{SRF}|_{X_0}^2}
{ \int_{X_0\times \{y\} } \eta_1 \omega_{SRF}|_{X_0}^2    }
\]
Here we notice that $\eta_1$ is equal to 1 on most of the measure of the fibres, so the denominator is approximately 1.
The function $\chi_0' f_0'$ is well defined on the support of $\chi_0'$, and has a weighted H\"older bound $\norm{ \chi_0' f_0'}_{C^{0,\alpha}_{\delta-2, \tau-2, t} } \leq C \norm{ f}_{C^{0,\alpha}_{\delta-2, \tau-2,t }(X)  }$, when $-2+\alpha<\delta<0$ and $-2+\alpha<\tau<0$. (This follows from Lemma \ref{HolderXytimesCpushforward} but does not manifest its full strength. A refined statement is Lemma \ref{fYweightedHolderaverageproperty}.)

 Then we write
\[
f_{0,1}= \chi_0'\eta_1(  f-  {f_0'(y) } ), \quad f_{0,2}= \chi_0' \eta_2 (f-  {f_0'(y) } ).
\]
By construction $f_{0,1}$ and $f_{0,2}$ are supported where the cutoff functions are supported, their $C^{0,\alpha}_{\delta-2, \tau-2,t}$ norms are bounded in terms of the same norm on $f$ up to a bounded factor,
and $f_{0,1}$ has fibrewise average zero. This gives a decomposition
\[
\chi_0' f= f_{0,1}+ f_{0,2}+ \chi_0' f_0'.
\]
\end{itemize}

To summarize, we obtain $f_{0,1}, f_1, f_2, \ldots f_{N'}$ which have some fibrewise average zero property, a function $f_{0,2}$ which is supported near the nodal point, and a function on the base
$
f_Y=\chi_0' f_0'+ f_1' + \ldots + f_{N'}',
$
which has a weighted H\"older bound $  \norm{f_Y}_{C^{0,\alpha}_{\delta-2, \tau-2, t}(X) }\leq C  \norm{f}_{ C^{0, \alpha}_{\delta-2, \tau-2,t}(X)  }$ with constant independent of $t$, because even though we are summing over many terms, near any given fibre only a small number of them contribute. These give a decomposition of $f$:
\begin{equation}
f= f_{0,1}+ f_{0,2}+ f_1+ f_2 + \ldots + f_{N'} + f_Y.
\end{equation}
The strategy is then to use the harmonic analysis we developed on various model geometries to divide and conquer all these pieces, at least approximately.

We first deal with $f_1, f_2, \ldots, f_{N'}$, which are supported on the support of $\chi_i'$, so in particular are far away from the singular fibre. Via the trivialisations $G_{y_i'}$, we regard these as functions on $ X_{y_i'}\times \C$ with average zero, so we can apply the Green operator $P_{y_i'}$ provided by Lemma \ref{HolderXytimesC} to  solve the Poisson equation on $X_{y_i'}\times \C $. To put these local solutions back on $X$, we use the cutoff function $\tilde{\chi}_i'=\gamma_2( \frac{|y-y_i'|}{\Lambda_3 t^{1/2} }   )$, where $\Lambda_3$ is a large number to be fixed independent of $t$. We write 
\begin{equation}
P_1 f= \sum_{i=1}^{N'} \tilde{\chi}'_i P_{y_i'} f_i
\end{equation}
Notice the cutoff procedure makes the function well defined on $X$.

\begin{lem}\label{Greenawayfromsingularfibre}
Given  $\epsilon_3$, then we can choose $\Lambda_3$ sufficiently large, so that for sufficiently small $t$ with respect to all previous choices, 
\begin{equation*}
\norm{ \Lap_{\omega_t} P_1 f - \sum_{i=1}^{N'} f_i }_{ C^{0, \alpha}_{\delta-2, \tau-2, t}(X)  }\leq \frac{1}{100} \norm{f}_{C^{0, \alpha}_{\delta-2, \tau-2, t}(X)   } ,
\end{equation*}
and 
\[
\norm{  P_1 f  }_{ C^{2, \alpha}_{\delta, \tau, t}(X)  }\leq C(\epsilon_3, \alpha, \delta, \tau) \norm{f}_{C^{0, \alpha}_{\delta-2, \tau-2, t}(X)   } .
\]
\end{lem}

\begin{proof}
In this proof we shall not keep track of dependence on $\delta, \tau, \alpha, \epsilon_3$.
By Lemma \ref{HolderXytimesC} there is a uniform estimate,
\[
\norm{ P_{y_i'} f_i }_{C^{2, \alpha}(X_{y_i'} \times \C)  } \leq C \norm{  f_i }_{C^{0, \alpha}(X_{y_i'}\times \C)  },
\]
and a uniform exponential decay estimate with rate $m$ on $|P_{y_i'} f_i|$ and the higher derivatives, for $|y-y_i'|\gtrsim t^{1/2}$, namely outside the support of $\chi_i'$.

We examine the norm of $\Lap_{\omega_t} P_{y_i'} f_i- f_i$. The error comes from two sources: the cutoff at scale $|y-y_i'|\sim \Lambda_3 t^{1/2}$, and the deviation of $\omega_t$ from the product metric on $X_{y_i'}\times \C$. Because of the exponential decay, the cutoff error is of order $O(\exp{ (-m \Lambda_3 ) } \norm{f_i}_{C^{0,\alpha}}  )$. The metric deviation error is of order $O(t^{1/2}\Lambda_3 \norm{P_{y_i'} f}_{C^{2,\alpha}} ) $ using Proposition \ref{metricdeviationerror}. For $t$ sufficiently small with respect to all previous choices, the errors suppressed by a power of $t$ can be ignored, so
\[
\norm{ \Lap_{\omega_t} P_{y_i'} f_i -f_i }_{ C^{0,\alpha} (X_{y_i'} \times \C) } \leq C \exp{ (-m \Lambda_3 ) } \norm{f_i}_{ C^{0,\alpha}( X_{y_i'} \times \C)   }.
\]  
The norm on $C^{0, \alpha}(X_{y_i'}\times \C)$ differs from the $C^{0,\alpha}_{ \delta-2, \tau-2, t  }(X)$ norm by a factor of order $t^{\frac{1}{2}(\delta-\tau) }$ (\cf section \ref{WeightedHolderspacesonX}).  The effects of the weight factors cancel out to give
\[
\norm{ \Lap_{\omega_t} P_{y_i'} f_i -f_i }_{ C^{0,\alpha}_{\delta-2, \tau-2, t} (X) } \leq C e^{-m\Lambda_3} \norm{f_i}_{ C^{0,\alpha}_{\delta-2, \tau-2, t} ( X)   } \leq C e^{-m\Lambda_3} \norm{f}_{ C^{0,\alpha}_{\delta-2, \tau-2, t} ( X)    }.
\] 
Now we want to sum the contributions over $i$.
Our cutoff procedure ensures that each forcing term $f_i$ can only influence the points with $|y-y_i'|\lesssim \Lambda_3 t^{1/2}$, so each point $y$ only receives contributions from $O(\Lambda_3^2 )$ terms. Thus
\[
\norm{P_1 f -\sum_{i=1}^{N'} f_i }_{ C^{0,\alpha}_{\delta-2, \tau-2, t} (X) } \leq C \Lambda_3^2 e^{-m\Lambda_3} \norm{f}_{ C^{0,\alpha}_{\delta-2, \tau-2, t} ( X)    }.
\]
When $\Lambda_3$ is chosen to be sufficiently large, we can make this coefficient arbitrarily small. Once we fix this choice, the claims of the Lemma are clear.
\end{proof}

\begin{rmk}
The intuition is that by the exponential decay property, the forcing functions with fibrewise average zero have localised effect for the Poisson equation. They almost do not interact with the rest of the manifold, and that explains why their effects do not accumulate.  
\end{rmk}

Next, we deal with $f_{0,1}$, which is supported on the support of $\chi_0' \eta_1$ and thus can be regarded as a function on $X_0\times \C$, and has some fibrewise average zero property. We prepare some cutoff function $\tilde{\chi}_0'$  on $Y$, supported in $\{ |y|< \epsilon_3 \}$, is equal to 1 on $\{ |y|\leq \frac{3}{4} \epsilon_3  \} $, and \[|\nabla^k_{ \frac{1}{t} \omega_Y } \tilde{\chi}_0'|\leq C(k) (t^{1/2} \epsilon_3^{-1} )^k.\]
We also need a logarithmic cutoff function
\[
\tilde{\eta}_1= \gamma_1 ( \frac{ \log( \frac{r\Lambda_2^3}{ t^{1/10}+ t^{1/12 }\rho'{1/6 } }  ) }{ \log \Lambda_2  } ),
\]
which equals 1 on the support of $f_{0,1}$, and
whose gradient is supported in the range $\frac{1}{\Lambda_2^2} \leq \frac{ r}{ t^{1/10}+ t^{1/12 }\rho'{1/6 } } \leq \frac{1}{\Lambda_2}$, involving $O(\log \Lambda_2)$ dyadic scales. Its key property is that $\norm{ d \tilde{ \eta}_1 }_{ C^{k,\alpha}_ {-1, -1, t } }  \leq  \frac{C}{ \log \Lambda_2}  $, namely that we can gain an extra factor of order $O( \frac{1}{ \log \Lambda_2})$ compared to ordinary cutoff functions. Then we set
\begin{equation}
P_{0,1 } f= \tilde{\eta}_1 P_{y=0} f_{0,1} ,
\end{equation}
where we recall from Lemma \ref{HolderXytimesCweighted2} that $P_{y=0}$ is the Green operator on $X_0\times \C$ for the product metric.

Similarly, we deal with $f_{0,2}$, which is supported on the support of $\chi_0' \eta_2$, so can be regarded as a function on $\C^3$. We need an extra logarithmic cutoff function 
\[
\tilde{\eta}_2= \gamma_2 ( \frac{ 2\log( \frac{r\Lambda_2^{3/2}}{ 2 (t^{1/10}+ t^{1/12 }\rho'{1/6 }) }  ) }{ \log \Lambda_2  } ),
\]
which equals 1 on the support of $f_{0,2}$, and whose gradient is supported in the range $\frac{2}{\Lambda_2} \leq \frac{r}{  t^{1/10}+ t^{1/12 }\rho'{1/6 } }\leq \frac{2}{\Lambda_2^{1/2}}\ll 1 $. Its key property is that
$\norm{ d \tilde{ \eta}_2 }_{ C^{k,\alpha}_ {-1, -1, t } }  \leq  \frac{C}{ \log \Lambda_2}  $. Then we set
\begin{equation}
P_{0,2} f= (\frac{t}{2A_0})^{1/3} \tilde{\eta}_2 P_{\C^3 } f_{0,2},
\end{equation}
where we recall from Proposition \ref{HolderC3} that $P_{\C^3}$ is the Green operator on $(\C^3, \omega_{\C^3})$. The scaling factor is inserted to account for the relation between $\omega_t$ and $\omega_{\C^3}$, so that $\Lap_{\omega_t} P_{0,2} f$ is approximately $f_{0,2}$.

\begin{lem}\label{Greennearsingularfibre}
We can choose $\Lambda_2\gg1$,  $\epsilon_3\ll1$  subject to the constraints in section \ref{Metricdeviation}, such that for $t$ sufficiently small with respect to all previous choices, 
\[
\begin{split}
&\norm{ \Lap_{\omega_t} P_{0,1} f - f_{0,1} }_{ C^{0, \alpha}_{\delta-2, \tau-2, t}(X)  }\leq \frac{1}{100} \norm{f}_{C^{0, \alpha}_{\delta-2, \tau-2, t}(X)   } , \\
&\norm{  P_{0,1} f  }_{ C^{2, \alpha}_{\delta, \tau, t}(X)  }\leq C(\delta, \tau, \alpha)\norm{f}_{C^{0, \alpha}_{\delta-2, \tau-2, t}(X)   },
\end{split}
\]
and morever
\[
\begin{split}
&\norm{ \Lap_{\omega_t} P_{0,2} f - f_{0,2} }_{ C^{0, \alpha}_{\delta-2, \tau-2, t}(X)  }\leq \frac{1}{100} \norm{f}_{C^{0, \alpha}_{\delta-2, \tau-2, t}(X)   } , \\
&\norm{  P_{0,2} f  }_{ C^{2, \alpha}_{\delta, \tau, t}(X)  }\leq C(\delta, \tau, \alpha)\norm{f}_{C^{0, \alpha}_{\delta-2, \tau-2, t}(X)   }.
\end{split}
\]
\end{lem}

\begin{proof}
The norm  $ \norm{\Lap_{\omega_t} P_{0,1} f- f_{0,1} }_{  C^{0, \alpha}_{\delta-2, \tau-2}(X_0\times \C)   } $  is caused by the cutoff error and the  deviation of $\omega_t$ from the product metric on $X_0\times \C$. The cutoff error caused by $\tilde{\chi}_0'$ is suppressed by a power of $t$, hence negligible. The cutoff error caused by $\tilde{\eta}_1$ is of order $O( \frac{1}{\log \Lambda_2} \norm{ P_{y=0} f_{0,1}}_{ C^{2, \alpha}_{\delta, \tau}(X_0\times \C)   } )$. The metric deviation error is of order $O( \epsilon_3^{1/7} \norm{ P_{y=0} f_{0,1}}_{ C^{2, \alpha}_{\delta, \tau}(X_0\times \C)   }  )$ using Proposition \ref{metricdeviationerror}, provided the constraints in section \ref{Metricdeviation} are satisfied. We also observe that, because our cutoff scale takes place far above the quantisation scale when $t$ is sufficiently small, the $C^{0, \alpha}_{\delta-2, \tau-2}(X_0\times \C)$ norm and the $C^{0, \alpha}_{\delta-2, \tau-2,t}(X)$ norm are equivalent for the functions under consideration. Combining these discussions,
\[
\begin{split}
\norm{\Lap_{\omega_t} P_{0,1} f- f_{0,1} }_{  C^{0, \alpha}_{\delta-2, \tau-2, t}(X)   }&\leq C( \frac{1}{\log \Lambda_2} +\epsilon_3^{1/7} ) \norm{ P_{y=0} f_{0,1}}_{ C^{2, \alpha}_{\delta, \tau}(X_0\times \C)   }   \\
&\leq C( \frac{1}{\log \Lambda_2} +\epsilon_3^{1/7} ) \norm{  f_{0,1}}_{ C^{0, \alpha}_{\delta-2, \tau-2}(X_0\times \C)   }  \\
& \leq C( \frac{1}{\log \Lambda_2} +\epsilon_3^{1/7} ) \norm{  f_{0,1}}_{ C^{0, \alpha}_{\delta-2, \tau-2, t}(X)   } \\
& \leq C( \frac{1}{\log \Lambda_2} +\epsilon_3^{1/7} ) \norm{  f}_{ C^{0, \alpha}_{\delta-2, \tau-2, t}(X)   } .
\end{split}
\]
Now setting $\Lambda_2\gg1$ and $\epsilon_3\ll1$ subject to the constraints in section \ref{Metricdeviation}, we can make the coefficient arbitrarily small. The first couple of claims follow.

The claims about $P_{0,2} f$ are almost completely analogous, except that we need to insert some scaling factors such as $t^{\frac{1}{6}(\delta-2)}$ which account for the difference between the $C^{0, \alpha}_{ \delta-2, \tau-2, t }(X)$ norm and the $C^{0, \alpha}_{\delta-2, \tau-2}(\C^3, \omega_{\C^3})$ norm.
\end{proof}

Now we consider the function $f_Y$  on the base $Y$. We will evantually reduce the question of inverting the Laplacian on $f_Y$ to a question essentially on the base. We first summarize the information about $f_Y$.

\begin{lem}\label{fYweightedHolderaverageproperty}
The function $f_Y$ satisfies the bound
\[
|f_Y(y)| \leq C\norm{f}_{ C^{0,\alpha}_{\delta-2, \tau-2, t}(X) } 
\begin{cases}
 (1+ t^{-1/2}|y|  )^{\delta-\tau}, & |y|< \epsilon_1, \\
 t^{ \frac{\tau-\delta}{2}   },
 & |y|\geq \epsilon_1,
 \end{cases}
\]
and for $|y-y'|\leq t^{1/2}$, there is a weighted H\"older bound
\[
|f_Y(y)- f_Y(y')|  \leq C\norm{f}_{ C^{0,\alpha}_{\delta-2, \tau-2, t}(X) } ( t^{-1/2} |y-y'|   )^\alpha
\begin{cases}
(1+ t^{-1/2}|y|  )^{\delta-\tau}, & |y|< \epsilon_1, \\
t^{ \frac{\tau-\delta}{2}   },
& |y|\geq \epsilon_1.
\end{cases}
\]
Morever, the average values
\[
\bar{f}= \frac{ \int_X f \omega_t^3 }{ \int_X \omega_t^3  }, \quad 
\overline{f_Y}= \frac{ \int_Y f \tilde{\omega}_Y }{ \int_Y \tilde{\omega}_Y  }=\int_Y f \tilde{\omega}_Y
\]
differ by only a `small amount':
\[
|\bar{f}- \overline{f_Y}| \leq C t^{\frac{1}{2}(\tau-\delta)  } \epsilon_3^{ \delta-\tau+2 } \norm{f}_{ C^{0,\alpha}_{\delta-2, \tau-2, t}(X)   }.
\]
In particular if $\bar{f}=0$ then $|\overline{f_Y}| \leq C t^{\frac{1}{2}(\tau-\delta)  } \epsilon_3^{ \delta-\tau+2 } \norm{f}_{ C^{0,\alpha}_{\delta-2, \tau-2, t}(X)   }.$ Here constants only depend on $\delta, \tau, \alpha$, and $t$ is assumed to be sufficiently small as usual.
\end{lem}

\begin{proof}
The weighted H\"older control on $f_Y$ is immediate from its construction and Lemma \ref{HolderXytimesCpushforward}.

We explain why the average values are close to each other. On the set $\{|y|>\epsilon_3\}$ which has most of the $\omega_t^3$-measures, the measure $\omega_t^3$ is very close to $\frac{3}{t} \omega_{SRF}|_{X_y}^2 \tilde{\omega}_Y$ up to some relative error suppressed by a power of $t$, and $f_Y$ is essentially just the integration along fibres using the $\omega_{SRF}|_{X_y}^2$ measure, so
\[
\int f\omega_t^3\sim \int \frac{3}{t} \tilde{\omega}_Y \int_{X_y} f \omega_{SRF}|_{X_y}^2 \sim \int \frac{3}{t} f_Y \tilde{\omega}_Y , \quad \int \omega_t^3 \sim \int \frac{3}{t} \tilde{\omega}_Y,
\]
and cancelling factors gives $\bar{f}- \overline{f_Y}\sim 0$.
After neglecting the small error caused by the set $\{|y|>\epsilon_3  \}$, and making a very crude estimate for the error 
 caused by the set $\{|y|\leq \epsilon_3 \}$, we arrive at 
\[
|\bar{f}- \overline{f_Y}| \leq C \frac{  \int_{ |y|\leq \epsilon_3 } |f|\omega_t^3 }{ \int_X \omega_t^3   } \leq C \frac{  \int_{ |y|\leq \epsilon_3} ( 1+ t^{-1/2} |y|  )^{\delta-\tau} \tilde{\omega}_Y  } { \int_Y    \tilde{\omega}_Y  }\norm{f}_{ C^{0,\alpha}_{\delta-2, \tau-2, t}(X)   }
\]
where we have integrated along fibres and cancelled factors.
The RHS is estimated by $C t^{\frac{1}{2}(\tau-\delta)  } \epsilon_3^{ \delta-\tau+2 } \norm{f}_{ C^{0,\alpha}_{\delta-2, \tau-2, t}(X)   }$.
\end{proof}

A requirement for solving the Poisson equation is that the forcing term has zero average. Thus we would like to decompose $f_Y$ into an average zero function and a small error function. Let $\chi_0''$ be a cutoff function on $Y$ with support in $\{|y|> \epsilon_1  \}$, and equals 1 on a subset of $Y$ with
at least half of the $\omega_Y$-measure, and has 
 bounded higher derivatives with respect to $\frac{1}{t}\omega_Y$.
Let
\[
f_Y' = f_Y - { \overline{f_Y} } \frac{ \chi_0''} { \int_Y \chi_0'' {\tilde{\omega}_Y}    }, \quad f_Y''={ \overline{f_Y} } \frac{ \chi_0''} { \int_Y \chi_0'' {\tilde{\omega}_Y}    }.
\]
By construction $\int_Y f_Y'\tilde{\omega}_Y=0$, the function $f_Y'$ satisfies the same kind of weighted H\"older estimate as $f_Y$ described in Lemma \ref{fYweightedHolderaverageproperty}, and if $\bar{f}=0$, then
\[
\norm{ f_Y'' }_{ C^{0,\alpha}_{\delta-2, \tau-2, t} (X) } \leq C|\overline{f_Y}| t^{\frac{1}{2}(\delta-\tau)  } \leq C \epsilon_3^{\delta-\tau+2} \norm{f}_{ C^{0,\alpha}_{\delta-2, \tau-2, t} (X) }.
\]
We now define $P_Y f$ to be the unique solution to the following Poisson equation on $Y$,
\[
\Lap_{ \frac{1}{t} \tilde{\omega}_Y  } (P_Y f  )= f_Y', \quad \int_Y (P_Y f) \tilde{\omega}_Y=0,
\]
or equivalently,
\begin{equation}
\sqrt{-1} \partial \bar{\partial} ( P_Y f )= \frac{1}{2t} f_Y' \tilde{\omega}_Y,   \quad \int_Y (P_Y f) \tilde{\omega}_Y=0.
\end{equation}

\begin{lem}\label{GreenY}
Assume that the average value $\bar{f}=0$. If $\epsilon_3$ is chosen to be sufficiently small, then for sufficiently small $t$ with respect to all previous choices,
\[
\norm{ \Lap_{\omega_t} P_Y f- f_Y}_{ C^{0,\alpha}_{\delta-2, \tau-2,t}(X)  } \leq \frac{1}{100}\norm{ f}_{ C^{0,\alpha}_{\delta-2, \tau-2,t}(X)  } ,
\]
and
\[
\norm{ \partial \bar{\partial} (P_Y f ) }_{ C^{0,\alpha} _{\delta-2, \tau-2,t}(X)    } \leq C(\delta, \tau, \alpha) \norm{  f} _{ C^{0,\alpha}_{\delta-2, \tau-2,t}(X)    }.
\]
\end{lem}

\begin{proof}
The estimate on $\norm{ \partial \bar{\partial} (P_Y f ) }_{ C^{0,\alpha} _{\delta-2, \tau-2,t}(X)    }$ is immediate from the definition of $P_Y f$ and the fact that $ \norm{  \frac{1}{t} \tilde{\omega}_Y }_{ C^{0,\alpha}_{0, 0, t  }(X)  }\leq C$.

To estimate $\norm{ \Lap_{\omega_t} P_Y f- f_Y}_{ C^{0,\alpha}_{\delta-2, \tau-2,t}(X)  } $, it suffices to control the norm $\norm{ \Lap_{\omega_t} P_Y f- f_Y'}_{ C^{0,\alpha}_{\delta-2, \tau-2,t}(X)  }$, since $\norm{f_Y'' }=O( \epsilon_3^{\delta-\tau+2} \norm{f} )$ can be made arbitrarily small by choosing small $\epsilon_3$. By a standard formula of the Laplacian,
\[
\sqrt{-1} \partial \bar{\partial} (P_Y f)\wedge \omega_t^2= \frac{1}{6} (\Lap_{\omega_t}P_Y f) \omega_t^3.
\]
The point is that to control the Laplacian, we do not actually need to estimate $P_Y f$; all we need is the tautological properties of $\sqrt{-1}\partial \bar{\partial } P_Y f$. From this formula and the definition of $P_Y f$,
\[
\Lap_{\omega_t}P_Y f=  f_Y'
\frac{ 3\tilde{\omega}_Y \wedge \omega_t^2 }{ t \omega_t^3  }.
\]

In the region $\{|y|\gtrsim t^{ \frac{6}{14+\tau}  } \}\cup \{ |y|\leq t^{ \frac{6}{14+\tau}  }, r\gtrsim t^{1/10}+ t^{1/12}\rho'^{1/6} \}$ where the semi-Ricci-flat behaviour is dominant, the quantity $\frac{ 3\tilde{\omega}_Y \wedge \omega_t^2 }{ t \omega_t^3  }$ is very close to 1. Following almost the same calculations in Lemma \ref{omegaterror1}, one can extract an estimate in this region
\[
\norm{ \frac{ \tilde{\omega}_Y \wedge \omega_t^2 }{  \sqrt{-1} \Omega\wedge \overline{\Omega} } -1}_{ C^{0,\alpha}_{0, 0,t}   } \leq Ct^{  \frac{\tau+2}{ 2(14+\tau)  }   }, \quad \norm{ \frac{  t \omega_t^3 }{ 3 \sqrt{-1} \Omega\wedge \overline{\Omega} } -1}_{ C^{0,\alpha}_{0, 0,t}   } \leq Ct^{  \frac{\tau+2}{ 2(14+\tau)  }   },
\]
so $\norm{ \frac{ 3\tilde{\omega}_Y \wedge \omega_t^2 }{ t \omega_t^3  }-1} _{  C^{0,\alpha}_{0, 0,t}   }\leq Ct^{  \frac{\tau+2}{ 2(14+\tau)  }   }$, hence in this region
\[
\norm{\Lap_{\omega_t} P_Y f- f_Y' }_{C^{0,\alpha}_{\delta-2, \tau-2,t}  } \leq Ct^{  \frac{\tau+2}{ 2(14+\tau)  }   } \norm{ f_Y'}_{ C^{0,\alpha}_{\delta-2, \tau-2,t} } \leq Ct^{  \frac{\tau+2}{ 2(14+\tau)  }   } \norm{ f}_{ C^{0,\alpha}_{\delta-2, \tau-2,t} }.
\]
This term is suppressed by a power of $t$, hence negligible.

In the region $\{ |y|< t^{\frac{6}{14+\tau}  }, r< t^{1/10}+t^{1/12}\rho'^{1/6}  \}$, we have a very coarse estimate 
\[
\norm{ \frac{ 3\tilde{\omega}_Y \wedge \omega_t^2 }{ t \omega_t^3  }-1} _{  C^{0,\alpha}_{0, 0,t}   }\leq C,
\]
hence
\[
\norm{\Lap_{\omega_t} P_Y f- f_Y' }_{C^{0,\alpha}_{\delta-2, \tau-2,t}  } \leq C \norm{ f_Y' }_{C^{0,\alpha}_{\delta-2, \tau-2,t}  } = C \norm{ f_Y }_{C^{0,\alpha}_{\delta-2, \tau-2,t}  }.
\]
The last equality is because $f_Y'=f_Y$ in this region. But the  bounds on $f_Y$ in Lemma \ref{fYweightedHolderaverageproperty} imply that 
\[
\begin{split}
\norm{ f_Y}_{ C^{0,\alpha}_{\delta-2, \tau-2,t}( \{ |y|< t^{\frac{6}{14+\tau}  }, r< t^{1/10}+t^{1/12}\rho'^{1/6}  \} )  }
&\leq C\norm{f}_{ C^{0,\alpha}_{\delta-2, \tau-2,t}(X)    } \sup_r \{ r^{2-\tau}, r^{2-\delta} \} \\
&\leq C\norm{f}_{ C^{0,\alpha}_{\delta-2, \tau-2,t}(X)    } t^{ \frac{2}{14+\tau}   }.
\end{split}
\]
Conceptually, 
this extra gain of $t^{ \frac{2}{14+\tau}   } $ factor comes from the fact that integration on fibre has a regularising effect for a certain range of weights. This suppression factor makes $\norm{\Lap_{\omega_t} P_Y f- f_Y' }_{C^{0,\alpha}_{\delta-2, \tau-2,t}  } $ negligible.
\end{proof}

\begin{rmk}
It is easy to see that 
\[
\norm{P_Y f}_{C^{2,\alpha}(Y) }\leq C \norm{f}_{  C^{0,\alpha}(Y)   },
\]
from which we have some $t$-dependent bound
\[
\norm{P_Y f}_{C^{2,\alpha}_{\delta,\tau,t}(X) }\leq C(t) \norm{f}_{  C^{0,\alpha}_{\delta-2, \tau-2, t}(X)   }.
\]
The subtlety in the Lemma is that we did not seek a fully $t$-independent estimate of $P_Y f$, but were content with estimating $\sqrt{-1}\partial \bar{\partial} P_Y f $. Our first reason for doing so is that geometrically speaking, metric quantities are primary, and the potential is only auxiliary. The second reason is that the weighted H\"older space for $P_Y f$ is quite awkward to work with, due to the fact that $\delta-\tau+2>0$. 
\end{rmk}

We now sum over the pieces to define an approximate Green operator 
\begin{equation}
Pf= P_{1} f + P_{0,1} f+ P_{0,2} f+ P_Y f
\end{equation}
Combining Lemma \ref{Greenawayfromsingularfibre}, \ref{Greennearsingularfibre}, \ref{GreenY} and fixing $\epsilon_3\ll1, \Lambda_2\gg1, \Lambda_3\gg1$ to satisfy all the constraints, we see

\begin{cor}
Let $-2+\alpha<\delta<0$, $-2+\alpha<\tau<0$, and assume $\delta$ avoids a discrete set of values. Let $f$ be a function in $C^{0,\alpha}_{\delta-2, \tau-2,t}(X)  $ with $\int_X f\omega_t^3=0 $. Then one can choose $\epsilon_3, \Lambda_2, \Lambda_3$ such that
\[
\norm{ \Lap_{\omega_t} P f- f}_{ C^{0,\alpha}_{\delta-2, \tau-2,t}(X)  } \leq \frac{1}{25}\norm{ f}_{ C^{0,\alpha}_{\delta-2, \tau-2,t}(X)  } ,
\]
and
\[
\norm{ \partial \bar{\partial} P f}_{ C^{0,\alpha}_{\delta-2, \tau-2,t}(X)  } \leq C(\delta, \tau, \alpha)\norm{ f}_{ C^{0,\alpha}_{\delta-2, \tau-2,t}(X)  } .
\]
\end{cor}

Finally, we prove Proposition \ref{Greenoperator}.

\begin{proof}
(Proposition \ref{Greenoperator})
Let $f$ be a function in $C^{0,\alpha}_{\delta-2, \tau-2, t}(X)$ with integral zero.
We can define 
\[
u= P\sum_{j=0}^\infty ( 1-\Lap_{\omega_t} P  )^j   f,
\]
which converges because $\norm{\Lap_{\omega_t} P-1}\leq \frac{1}{25}$. Although we do not have good control on $u$ directly, we do know 
\[
\norm{ \partial \bar{\partial} u}_{ C^{0, \alpha}_{\delta-2, \tau-2, t}  }\leq C\norm{\partial \bar{\partial} P} \norm{ f} \leq C\norm{f}.
\]

 %Consider the sequence of approximate solutions of the Poisson equation
%\[
%u_k=P \sum_{j=0}^k ( 1-\Lap_{\omega_t} P  )^j   f,
%\]
%which satisfy
%\[
%\norm{\Lap_{\omega_t} u_k -f}_{C^{0,\alpha}_{\delta-2, \tau-2, t}(X)  } = \norm{(1-\Lap_{\omega_t} P)^{k+1} f}_{C^{0,\alpha}_{\delta-2, \tau-2, t}  } \leq (\frac{1}{25})^{k+1} \norm{f}_{ C^{0,\alpha}_{\delta-2, \tau-2, t}   },
%\]
%\[
%\norm{u_k}_{  C^{2,\alpha}_{\delta, \tau, t}(X)    } \leq C\norm{P}\norm{f}_{ C^{0,\alpha}_{\delta-2, \tau-2, t}(X)   },
%\]
%and
%\[
%\norm{ \partial \bar{\partial} u_k}_{C^{0,\alpha}_ {\delta-2, \tau-2, t}(X) }   \leq C\norm{ \partial \bar{\partial} P} \norm{f }_{ C^{0,\alpha}_ {\delta-2, \tau-2, t}  } \leq C\norm{f}_{ C^{0,\alpha}_ {\delta-2, \tau-2, t}  }.
%\]
%The sequence $u_k$ is Cauchy, so converges to a limit $u$, with $\Lap_{\omega_t} u= f$. The $C^{2,\alpha}_{\delta, \tau, t }(X)$ norm of $u$ may be fairly large, depending on the operator norm of $P$, which is a finite number dependent on $t$. However, the uniform bound on $\norm{\partial \bar{\partial } u_k}$ implies that $\norm{\partial \bar{\partial} u}\leq C\norm{f}  $ independent of $t$. 
Using our convention for the Laplacian, the $\partial\bar{\partial}$ exact (1,1)-form $\theta= {2\sqrt{-1}} \partial \bar{\partial} u$ solves $\Tr_{\omega_t} \theta=\Lap_{\omega_t}u= f$, so setting 
\begin{equation}\label{constructionofR}
\mathcal{R} (f)=\theta =2\sqrt{-1} \partial \bar{\partial } P \sum_{j=0}^\infty (1-\Lap_{\omega_t} P)^j f
\end{equation}
gives the desired right inverse $\mathcal{R}$ with bounds. We remark that $\mathcal{R}$ maps a real valued function to a real (1,1)-form.
\end{proof}

\subsection{Exponential localising property}\label{Exponentiallocalising}

An interesting consequence of the parametrix construction in section \ref{Decompositionpatchingparametrix} is that away from the singular fibres, the effect of the forcing term is very localised.

\begin{prop}
In the setup of Proposition \ref{Greenoperator}, if the forcing function $f$ is supported away from the neighbourhood of a smooth fibre
$\{ x\in X: |\pi(x)-y'|\leq d   \}\subset X\setminus \{ |y|\leq \epsilon_1   \}$, then 
\[
\norm{ \mathcal{R} f}_{ C^{0,\alpha}_{\delta-2, \tau-2,t} (\{ x\in X: |\pi(x)-y'|\leq \frac{d}{2}   \} ) } \leq C(\delta, \tau, \alpha) \exp{( \frac{ - m(\delta, \tau, \alpha) d}{ t^{1/2} } ) } \norm{ f}_{ C^{0,\alpha}_{\delta-2, \tau-2,t}(X) },
\]
where the exponential decay constants do not depend on $f, d, y', t$.
\end{prop}

\begin{proof}
Without loss of generality the set $\{ |y-y'|\leq d  \}\subset Y$ has less than $\frac{1}{3}$ of the total $\omega_Y$-measure.
Using the flexibility of the definition of the cutoff function $\chi_0''$, we may arrange that $\chi_0''$ is supported away from $\{ |y-y'|\leq d  \}\subset Y$. 
Then in the construction of the operator $1-\Lap_{\omega_t} P$, only the part $1-\Lap_{\omega_t} P_1$ can propagate a forcing term supported outside $\{ x\in X: |y-y'|\leq d  \}\subset X$ into this set.

The key point now is that the construction of $\mathcal{R}$ in (\ref{constructionofR}) involves an iteration of the operator $1-\Lap_{\omega_t} P$, and each iteration can propagate the support towards $X_{y'}$ only by a very small amount of $\omega_Y$-distance, of order $\sim O(\Lambda_3 t^{1/2} )$. Thus it requires about $O( \frac{d}{\Lambda_3 t^{1/2}  }  )$ iterations to propagate the forcing term into the set $\{ |y-y'|\leq \frac{d}{2}  \}\subset X$. Each iteration results in a damping factor $\frac{1}{25}$ on the norm of the forcing term, so 
\[
\norm{ \mathcal{R} f}_{ C^{0,\alpha}_{\delta-2, \tau-2,t} (\{ x\in X: |\pi(x)-y'|\leq \frac{d}{2}   \} ) } \lesssim  (\frac{1}{25})^{ \frac{ - C d}{ t^{1/2} }  } \norm{ f}_{ C^{0,\alpha}_{\delta-2, \tau-2,t}(X) }
\]
as required.
\end{proof}

\section{Collapsed Calabi-Yau metric}

\subsection{Perturbation to the Calabi-Yau metric}

We carry out the main gluing construction. To avoid complication, we choose the weights $\tau=-\frac{2}{3}$ and $-\frac{3}{13}<\delta<0$ (\cf Remark \ref{HolderimpliesLinfty}). We also assume that $\delta$ avoids a discrete set of values, so Proposition \ref{Greenoperator} applies.
The elementary numerical properties of these weights are summarised as
\[
\begin{cases}
\norm{f }_{ C^{0,\alpha}_{ 0, 0, t}(X)  }
\leq
C t^{ \frac{1}{6}\delta-\frac{1}{3}  }\norm{f }_{ C^{0,\alpha}_{ \delta-2, \tau-2, t}(X)  }, \\
\norm{fg }_{ C^{0,\alpha}_{\delta-2, \tau-2, t}(X)  }
\leq
C\norm{g} _{ C^{0,\alpha}_{ 0, 0, t}(X)  }
\norm{f }_{ C^{0,\alpha}_{ \delta-2, \tau-2, t}(X)  },
\end{cases}
\]
where $f, g$ are arbitrary elements of the function spaces. The same statement holds for forms.

\begin{thm}\label{CYmetric}
Let $\tau=-\frac{2}{3}$ and $-\frac{3}{13}< \delta< 0$. Assume $\delta$ avoids the discrete set of values in Proposition \ref{Greenoperator}. Then for sufficiently small $t$, the Calabi-Yau metric $\tilde{\omega}_t$ in the class $[\omega_t]= [\omega_X]+ \frac{1}{t}[\omega_Y]$ is close to $\omega_t$, with the bound
\begin{equation}
\norm{ \tilde{\omega}_t- {\omega}_t  }_{ C^{0, \alpha}_{\delta-2, \tau-2, t}(X)  }\leq C(\delta,  \alpha) t^{ \frac{23}{60} + \frac{1}{20}\delta  }.
\end{equation}
In particular by the numerical property of the weights
\[
\norm{ \tilde{\omega}_t- {\omega}_t  }_{ C^{0, \alpha}_{0, 0, t}(X)  }\leq C(\delta,  \alpha) t^{ \frac{1}{20} + \frac{13}{60}\delta  }\ll1.
\]
\end{thm}

\begin{proof}
We wish to find the Calabi-Yau metric $\tilde{\omega}_t$ as a small perturbation of the metric ansatz $\omega_t$. Using the defining conditions (\cf section \ref{Regularisingthesemiflatmetric})
\[
\tilde{\omega}_t^3 = a_t \sqrt{-1}\Omega \wedge \overline{\Omega}, \quad \omega_t^3= a_t (1+ f_t) \sqrt{-1} \Omega \wedge \overline{\Omega},
\]
it suffices to find a function $f\in C^{0, \alpha}_{\delta-2, \tau-2, t  }(X)$ with $\int_{X} f\omega_t^3=0$, such that 
\begin{equation}\label{perturbationtoCY}
( \omega_t +  \mathcal{R} f  )^3= \frac{1}{1+f_t} \omega_t^3, \quad \norm{f}_{ C^{0, \alpha}_{\delta-2, \tau-2, t  }(X) } \leq  C t^{ \frac{23}{60} + \frac{1}{20}\delta  } .
\end{equation}
Once we found $f$, then setting $\tilde{\omega}_t= \omega_t+  \mathcal{R} f $ gives the CY metric with estimate. We point out the subtlety that although $\omega_t$ and $f$ can fail to be smooth, the equation $\tilde{\omega}_t^3=a_t \sqrt{-1}\Omega \wedge \overline{\Omega}$ would imply the smoothness of $\tilde{\omega}_t$ by standard argument, so $\tilde{\omega}_t$ is an honest CY metric, and hence the unique CY metric in this class.

We now focus on solving (\ref{perturbationtoCY}). Define the nonlinear operator $\mathcal{F}$ acting on the subspace
$\{ f: \int_X f\omega_t^3=0  \}\subset C^{0,\alpha}_{\delta-2, \tau-2,t }(X)$,
\[
\mathcal{F}(f)= \frac{ ( \omega_t +  \mathcal{R} f  )^3 }{   \omega_t ^3 }-1.
\]
Notice $\mathcal{F}(f)$ automatically has zero integral.
We restrict attention to the open subset
\[
\mathcal{U}=\{ f\in C^{0, \alpha}_{\delta-2, \tau-2,t}(X) : \int_X f\omega_t^3=0, \norm{f}_{ C^{0, \alpha}_{0,0, t}(X)  }< \epsilon_4    \} \subset C^{0, \alpha}_{\delta-2, \tau-2,t}(X)
\]
where $\epsilon_4\ll1$ is a small constant independent of $t$,
and we separate $\mathcal{F}$ into the linearisation and the nonlinearity,
\[
\mathcal{F}(f)= \Tr_{\omega_t} \mathcal{R} f+ Q(f)= f+ Q(f).
\]
Here $Q(0)=0$.
The equation (\ref{perturbationtoCY}) can be cast in the form of a fixed point equation
\[
f= \frac{-f_t}{1+ f_t}- Q(f).
\]

Using the basic numerical properties of the weights, for $u, v\in \mathcal{U}$, the nonlinearity $Q$ satisfies
\[
\begin{split}
\norm{Q(u)- Q(v)}_{ C^{0,\alpha}_{\delta-2, \tau-2, t}(X)  }& \leq C ( \norm{u}_{ C^{0, \alpha}_{0,0, t}(X)  } +  \norm{v}_{ C^{0, \alpha}_{0,0, t}(X)  }   )\norm{u-v}_{ C^{0,\alpha}_{\delta-2, \tau-2, t}(X)   }  \\
&\leq C\epsilon_4 \norm{u-v}_{C^{0,\alpha}_{\delta-2, \tau-2, t}(X)   } \\
& \leq \frac{1}{2} \norm{u-v}_{C^{0,\alpha}_{\delta-2, \tau-2, t}(X)   }.
\end{split}
\]
The contraction property in the last step can be ensured by choosing $\epsilon_4$ sufficiently small. On the other hand,
our volume error estimate (\ref{omegaterror}) reads
\[
\norm{f_t}_{  C^{0,\alpha}_{\delta-2, \tau-2, t}(X)     } \leq C t^{ \frac{23}{60} + \frac{1}{20}\delta  } ,
\]
which implies by the numerical properties of the weights, that
\[
\norm{ \frac{-f_t}{1+ f_t}  }_{  C^{0,\alpha}_{\delta-2, \tau-2, t}(X)     } \leq C t^{ \frac{23}{60} + \frac{1}{20}\delta  } , \quad \norm{ \frac{-f_t}{1+ f_t}  }_{  C^{0,\alpha}_{0, 0, t}(X)     } \leq C t^{ \frac{13}{60}\delta+ \frac{1}{20}  }\ll\epsilon_4.
\]
The Banach fixed point theorem then yields a solution $f\in \mathcal{U}$ to equation (\ref{perturbationtoCY}) with estimates, as required.
\end{proof}

\begin{rmk}
A particular consequence of 
 our bound is
\[
|\tilde{\omega}_t- \omega_t|_{\omega_t  } \leq C(\delta) t^{ \frac{1}{20}+ \frac{13}{60} \delta  }, \quad -\frac{3}{13}< \delta<0 .
\]
Notice that the exponent $0<\frac{1}{20}+ \frac{13}{60}\delta< \frac{1}{20}$ is rather small, which indicates that our metric ansatz $\omega_t$ is a rather coarse approximation. Our gluing construction is only made possible because of the optimal nature of the linear theory.
\end{rmk}

We now sketch the proof of Theorem \ref{GHlimitfibrescale} and \ref{omegaCblowuplimit}, which are immediate consequences of the main gluing theorem \ref{CYmetric}.

\begin{proof}
By construction $\omega_t$ is $C^0$ close to the product metric $G_0^*(\omega_{SRF}|_{X_0}+ \frac{A_0}{t} \sqrt{-1} dy\wedge d\bar{y})$ in the region $\{ r> t^{1/10}+ t^{1/12}\rho'^{1/6} , |y|< t^{ \frac{6}{14+\tau } }  \}$, with error suppressed by a power of $t$.
The region $\{  r< t^{1/10}+ t^{1/12}\rho'^{1/6} , |y|< t^{ \frac{6}{14+\tau } }  \}$ is negligible in the GH convergence because every point is within $\tilde{\omega}_t$ distance $t^{ \frac{1}{14+\tau}  }$ to the previous region. The region $\{ |y|> t^{ \frac{6}{14+\tau } } \}$ is invisible to the pointed GH limit because its $\tilde{\omega}_t$-distance to the nodal point is of order 
$O( t^{\frac{-(2+\tau)}{2(14+\tau)} }  )$, which diverges to infinity. Thus the pointed GH limit of $\tilde{\omega}_t$ is $X_0\times \C$ with the product metric, proving Theorem \ref{GHlimitfibrescale}.

On any fixed Euclidean ball inside $ F_t^{-1}(U_1)\subset \C^3$  centred at the origin (\cf section \ref{GeometryofthemodelmetricC3} for notation), the metric ansatz has the asymptotic formula $\omega_t\sim ( \frac{t}{2A_0} )^{1/3} \omega_{\C^3} $ as $t\to 0$, so Theorem \ref{CYmetric} easily implies Theorem \ref{omegaCblowuplimit} as well.
\end{proof}

\begin{rmk}
As a digression, the exponential localising property discussed in section \ref{Exponentiallocalising} implies that the CY metric near a smooth fibre is locally determined up to exponentially small corrections from the rest of the manifold, and in particular receives almost no correction effect from the initial errors supported near the singular fibre. Thus contrary to the low global regularity of our metric ansatz $\omega_t$, the actual CY metric $\tilde{\omega}_t$ may have much better regularity near a given smooth fibre, such as admitting a formal power series expansion in $t$ similar to the work of J. Fine \cite{Fine}. It is interesting to compare this observation with the very recent work of Hein and Tosatti \cite{HeinTosatti}.
\end{rmk}

\subsection{Open directions}

In this section we speculate how this work may be generalised, in the direction of describing collapsing CY metrics $\tilde{\omega}_t$ on more complicated fibrations $\pi:X\to Y$ 
over a 1-dimensional base $Y$, where $[\tilde{\omega}_t]$ lies in the K\"ahler class $[\omega_X+ \frac{1}{t}\omega_Y]$, and $0<t\ll1$. To begin with, we point out that our gluing construction depends essentially
on the fact that all fibres admit (possibly singular) Calabi-Yau metrics, and on the existence of the model CY metric $\omega_{\C^3}$ on $\C^3$, which fits well with the singularity in the fibration $\pi$. The 1-dimensional base assumption is also essential because much less is known about the generalised KE metrics on $Y$ for higher dimensions.

Following the papers \cite{Ronan}\cite{Gabor} a large class of examples of complete CY metrics on $\C^n$ are now known, which generalise the model metric $\omega_{\C^3}$.  In particular, on the total space of the standard higher dimensional Lefschetz fibration $f: \C^n\to \C$ where $f=\sum z_i^2$, there is a complete CY metric $\omega_{\C^n}$, whose asymptotic behaviour at infinity approximates the semi-Ricci-flat metric on $\C^n$, namely that in the fibre direction it approximates the Stenzel metrics on the fibres, and in the horizontal direction it is predominantly the pullback of the Euclidean metric on the base.

Now suppose a projective CY manifold $X$ admits a Lefschetz fibration $\pi: X\to Y$, where the fibres have complex dimension at least 3. By adjunction, the fibres are Calabi-Yau varieties in their own right, and admit (possibly singular) Calabi-Yau metrics. The result of Hein and Sun \cite{HeinSun} says in particular that the CY metrics on the singular fibres are modelled on the Stenzel metric near the nodal point, with polynomial rate of convergence. By standard gluing argument, the CY metrics on the smoothing fibres are modelled on the stenzel metric in the region close to the vanishing cycles. Thus it seems very plausible that the collapsing metric $\tilde{\omega}_t$ is obtained by gluing a suitably scaled copy of $\omega_{\C^n}$ to a suitably regularised version of the semi-Ricci-flat metric on $X$. It is also conceivable to extend this picture to more complicated fibrations with isolated critical points, using the model metrics provided by \cite{Ronan}\cite{Gabor}.

In a slightly different vein, the strategy of producing complete CY metrics in \cite{Ronan}\cite{Gabor} does not depend in an essential way on the ambient manifold being $\C^n$. To give a special interesting example to indicate the possible generalisation, it is a folklore speculation that there may be a non-standard complete CY metric $\omega_{Q_\epsilon}$ on the affine quadric $Q_{\epsilon}=\{ \zeta_1^2+ \zeta_2^2+ \zeta_3^2+ \zeta_4^2= \epsilon^2   \}$, which admits a Lefschetz fibration by projecting to the first coordinate
\[
f: Q_\epsilon \to \C, \quad (\zeta_1, \zeta_2, \zeta_3, \zeta_4)\mapsto \zeta_1.
\]
 There are exactly two singular fibres, corresponding to $\zeta_1=\pm \epsilon$. The expected behaviour is that asymptotically near infinity $\omega_{Q_\epsilon}$ looks like the semi-Ricci-flat metric on $Q_\epsilon$, namely that restricted to the fibres it approximates the Eguchi-Hanson metrics on the fibres, and in the horizontal direction it is dominated by the pullback of the Euclidean metric $\frac{\sqrt{-1}}{2} d\zeta_1\wedge d\bar{\zeta}_1$; in particular $\omega_{Q_\epsilon}$ has maximal volume growth rate and tangent cone $\C^2/ \Z_2\times \C$ at spatial infinity. As $\epsilon\to 0$, it is expected that $\omega_{Q_\epsilon}$ converges to a CY metric $\omega_{Q_0}$ on the conifold $Q_0$, with local tangent cone at the origin isometric to the Stenzel cone, and tangent cone at infinity 
isometric to $\C^2/\Z_2 \times \C$. This would be one of the simplest examples of such conjectural transition behaviours between different Calabi-Yau cones. Alternatively, these  $\omega_{Q_\epsilon}$ can be regarded as a family of metrics on the fixed complex manifold $Q_1$, by changing the relative size of the fibre compared to the base, much like our setup in the compact case.

Now the relevance of $\omega_{Q_\epsilon}$ to the collapsing metrics comes in when we  `collide two singular fibres'. More formally, let $X_\epsilon$ be a polarised family of projective CY 3-folds admitting Lefschetz K3 fibrations over $Y= \mathbb{P}^1$, such that in a neighbourhood containing two critical points, the fibration is modelled by $f: Q_\epsilon\to \C$. The collapsing CY metric $\tilde{\omega}_{t, \epsilon}$ depends on both the small collapsing parameter $t$ and the small degeneration parameter $\epsilon$ in the picture. 

\begin{itemize}
\item 
When $\epsilon>0$ is fixed and $t\to 0$, the situation is a collapsing Lefschetz K3 fibration, covered by the gluing construction	in this paper. The $\tilde{\omega}_{t, \epsilon}$ distance between the two nearby singular fibres is of order $O( \epsilon t^{-1/2}  )$. The quantisation scale, namely the $\tilde{\omega}_{t, \epsilon} $-length scale of the $\C^3$ bubble embedded in $X_\epsilon$, is of order $O((\frac{t}{A_0})^{1/6}  )$.
To understand how $A_0$ depends on $\epsilon$, we notice that for $\zeta_1$ very close to $\epsilon$, the Lefschetz fibration is approximately 
\[
y=-2\epsilon( \zeta_1-\epsilon)=  \zeta_2^2+ \zeta_3^2+ \zeta_4^2,
\]
which means
\[
A_0\sim \int_{X_{y=0}} \Omega_0\wedge \overline{\Omega}_0 \sim \int_{X_{y=0}} \frac{ \Omega\wedge \overline{\Omega} }{ dy\wedge d\bar{y} }\sim  \frac{1}{\epsilon^2} \int_{X_{y=0}} \frac{ \Omega\wedge \overline{\Omega} }{ d\zeta_1\wedge d\bar{\zeta}_1 }
\sim  O(\frac{1}{\epsilon^2}  ).
\]
Hence the quantisation scale is $O(t^{1/6} \epsilon^{1/3})$.
\\

\item
When we descrease $\epsilon$ until $\epsilon\sim t$,  then 
$O(\epsilon t^{-1/2}  )=O( t^{1/6} \epsilon^{1/3}  )=O( t^{1/2} )$, namely 
 the quantisation scale is comparable to the $\omega_{t, \epsilon}$-distance between the two critical points, so the interaction between the two $\C^3$ bubbles become significant. On the other hand, if we substitute $\zeta_i= \epsilon \zeta_i'$, and scale the metric by a factor $t^{-1}$ so that the distance scale becomes of order 1, then as $t\sim \epsilon\to 0$ we expect to see a blow up limit complex analytically isomorphic to $Q_1$, which is up to some scaling factor isometric to some member of the family of model metrics on $Q_1$.
 \\

\item 
When $\epsilon=0$, the CY manifolds develop a local conical singularity, and by Hein and Sun's result the CY metric $\tilde{\omega}_{t, 0}$ is locally modelled on the Stenzel cone, at least on some extremely small scale. When $\epsilon\ll t$, the effect of collapsing is insignificant in a very small region near the conical point, and one sees the usual behaviour of the smoothing of the conical singularity.

\end{itemize}

These discussions are meant to suggest that there is a numerous supply of non-compact complete CY metrics associated to a fibration structure, and these examples are intimately tied to the local behaviour of collapsing metrics on compact CY manifolds admitting fibration structures.

\end{document}